\documentclass[12pt,reqno]{amsart}
\usepackage[a4paper,bindingoffset=0.1in,left=1in,right=1in,top=1in,bottom=1in,footskip=.25in]{geometry}
\usepackage{indentfirst,amssymb,amsmath,amsthm}    
\usepackage{newtxtext,newtxmath}
\usepackage{setspace}
\usepackage{times}
\usepackage[utf8]{inputenc}
\usepackage[T1]{fontenc}
\usepackage{verbatim}
\usepackage{hyperref}
\hypersetup{colorlinks=true, linkcolor=blue,citecolor=blue, urlcolor=blue}
\DeclareMathOperator{\ord}{ord}
\DeclareMathOperator{\dime}{dim}
\DeclareMathOperator{\pideg}{PI-deg}

\DeclareMathOperator{\cf}{Fract}
\DeclareMathOperator{\gcdi}{gcd}
\DeclareMathOperator{\lcmu}{lcm}
\DeclareMathOperator{\diagonal}{diag}

\DeclareMathOperator{\tors}{tor}
\DeclareMathOperator{\colu}{col}
\DeclareMathOperator{\rowa}{row}

\DeclareMathOperator{\kere}{ker}
\DeclareMathOperator{\ran}{rank}

\DeclareMathOperator{\spect}{Spec}
\DeclareMathOperator{\mspect}{Maxspec}
\usepackage{mathtools}

\numberwithin{equation}{section}
 
\newtheorem{theo}{Theorem}[section]
\newtheorem{defi}[theo]{Definition}
\newtheorem{lemm}[theo]{Lemma}
\newtheorem{rema}[theo]{Remark}
\newtheorem{coro}[theo]{Corollary}
\newtheorem{prop}[theo]{Proposition}

\begin{document}

\setcounter{page}{1} 
\baselineskip .65cm 
\pagenumbering{arabic}

\title[PI Quantized Weyl Algebras]{Simple Modules and Azumaya Loci over \\the PI quantized Weyl Algebras}
\author [Sanu Bera~ And~ Snehashis Mukherjee]{Sanu Bera$^1$ \and Snehashis Mukherjee$^2$}

\address {\newline Sanu Bera$^1$ and Snehashis Mukherjee$^2$
\newline School of Mathematical Sciences, \newline Ramakrishna Mission Vivekananda Educational and Research Institute (rkmveri), \newline Belur Math, Howrah, Box: 711202, West Bengal, India.
 }
\email{\href{mailto:sanubera6575@gmail.com}{sanubera6575@gmail.com$^1$};\href{mailto:tutunsnehashis@gmail.com}{tutunsnehashis@gmail.com$^2$}}

\begin{abstract}
In this article, both versions of multiparameter quantum Weyl algebras have been studied at the roots of unity. The center, PI degree, maximal-dimensional simple modules, and Azumaya locus have been explicitly computed for such algebras.
\end{abstract}
\maketitle
\section{{Introduction}}
Let $\mathbb{K}$ be a field and $\mathbb{K}^*$ denotes the multiplicative group of non-zero elements of $\mathbb{K}$ and $n$ is a positive integer. Let $\Lambda:=\left(\lambda_{ij}\right)_{n \times n}$ be an  $n \times n$ multiplicatively antisymmetric matrix over $\mathbb{K}$, that is, $ \lambda_{ii}=1$ and $\lambda_{ij}\lambda_{ji}=1$ for all $1 \leq i,j\leq n$ and and let $\underline{q}:=(q_1,\cdots,q_n)$ be an $n$-tuple elements of $\mathbb{K}\setminus\{0,1\}$. Given such $\Lambda$ and $\underline{q}$, the Maltsiniotis multiparameter quantized Weyl algebra ${A_n^{\underline{q},\Lambda}}$, called {\it quantum Weyl algebra}, is the algebra generated over the field $\mathbb{K}$ by the variables $x_1,y_1,\cdots,x_n,y_n$ subject to the following relations: 
\begin{align*}
y_iy_j&=\lambda_{ij}y_jy_i \ \ \ \ \ \forall \ \ \ 1\leq i<j\leq n,\\
x_ix_j&=q_i\lambda_{ij}x_jx_i \ \ \ \ \ \forall \ \ \ 1\leq i<j\leq n,\\
x_iy_j&=\lambda_{ij}^{-1}y_jx_i\ \ \ \ \ \forall \ \ \ 1\leq i<j\leq n, \\
y_ix_j&=q_i^{-1}\lambda_{ij}^{-1}x_jy_i \ \ \ \ \ \forall \ \ \ 1\leq i<j\leq n,\\
x_iy_i-q_iy_ix_i&=1+\sum_{k=1}^{i-1}(q_k-1)y_kx_k \ \ \ \ \ \forall \ \ \ 1\leq i\leq n.
\end{align*}
The quantum Weyl algebra ${A_n^{\underline{q},\Lambda}}$, which arises from the work of Maltsiniotis \cite{gm} on noncommutative differential calculus, has been extensively studied in \cite{ad,gz,krg,daj}. When $n=1$, ${A_n^{\underline{q},\Lambda}}$ is a rank one quantized Weyl algebra $A_1^{q}$, whose prime ideals are classified in \cite{krg}. The automorphism group of $A_1^{q}$ was determined in \cite{ad}.  The prime spectrum and the automorphism group for ${A_n^{\underline{q},\Lambda}}$ in the generic case were studied in \cite{lr}. The isomorphism problem for ${A_n^{\underline{q},\Lambda}}$ is solved in \cite{gh2}.  
\par Another family of multiparameter quantized Weyl algebras has been studied in the literature, including \cite{aj}, which has more symmetric defining relations than those of ${{A}_n^{\underline{q},\Lambda}}$. The quantized Weyl algebra of symmetric type is denoted by ${\mathcal{A}_n^{\underline{q},\Lambda}}$ and is referred to as an alternative quantum Weyl algebra. The definition of the algebra ${\mathcal{A}_n^{\underline{q},\Lambda}}$ in terms of generators and relations has been mentioned in Section 7. 
\par We note that ${\mathcal{A}_n^{\underline{q},\Lambda}}$ and ${A_n^{\underline{q},\Lambda}}$ are closely related in many aspects \cite{aj, daj}. Both algebras have some iterated skew polynomial presentation twisted by automorphisms and derivations and also have a common localization. The prime ideals of ${\mathcal{A}_n^{\underline{q},\Lambda}}$ were classified in generic case \cite{aj}. When $\lambda_{ij}=1$, the algebra ${\mathcal{A}_n^{\underline{q},\Lambda}}$ is isomorphic to the tensor product $A_1^{q_1}\otimes\cdots \otimes A_1^{q_n}$, whose automorphism group has been studied in \cite{cpw} for $q_i\neq 1$. The automorphism group and isomorphism problem for ${\mathcal{A}_n^{\underline{q},\Lambda}}$ in the generic case have been studied in \cite{xt}. 
\par Most of these results concern the generic case when the algebras are not polynomial identity. In this article, we shall focus on quantized Weyl algebras when they are PI algebras (equivalently, defining parameters are roots of unity).  
\subsection*{Assumptions:}In the context of the root of unity setting, throughout this chapter we will be assuming the following assumption regarding the defining multiparameters ${q_i}$ and $\lambda_{ij}$ as follows:
\begin{equation}\label{asm}
\begin{minipage}{0.9\textwidth}
\begin{itemize}
    \item[(1)] $q_i$'s are primitive $l_i$-th root of unity for $1\leq i \leq n$.
    \item[(2)] $\lambda_{ij}$'s are $l_i$-th roots of unity for all $1\leq i\leq j\leq n$. 
    \item[(3)] the divisibility condition $l_1|l_2|\cdots|l_n$.
    \end{itemize}
\end{minipage} \tag{$\star$}
\end{equation}
As $q_i\in \mathbb{K}\setminus\{0,1\}$, it follows that $l_i\geq 2$ for all $1\leq i\leq n$. The assumptions (\ref{asm}) are satisfied in the important uniparameter cases:
\begin{itemize}
    \item [(A)] when $q_1=\cdots=q_n$ and $\lambda_{ij}=1$ (cf. \cite{bry}).
    \item [(B)] when $q_i=q^2$, $\lambda_{ij}=q^{-1}$ for $1\leq i<j\leq n$ and $\ord(q)$ is odd (cf. \cite{cg,ig}).
\end{itemize} 
If $q$ is a root of unity, say of order $l$, then the first quantum Weyl algebra $A_1^{q}$ is PI algebra with $\pideg(A_1^{q})=l$ and hence all irreducible representations of $A_1^{q}$ are finite dimensional and explicitly described up to equivalence in two research projects directed by E. Letzter \cite{blt} and by L. Wang \cite{he} via studying the matrix solutions $(X,Y)$ of the equation $xy-qyx=1$. To date, there has been no comprehensive classification of simple modules over the PI quantized Weyl algebras for $n\geq 2$. Moreover, Chelsea Walton posed a problem in \cite[Problem 2]{cw} concerning the explicit classification of irreducible representations of the quantum Weyl algebras in the uniparameter case (B), up to equivalence. In this chapter, our objective is to compute the PI degrees, classify all maximal dimensional simple modules and consequently determine the Azumaya locus for the multiparameter quantized Weyl algebras ${A_n^{\underline{q},\Lambda}}$ and ${\mathcal{A}_n^{\underline{q},\Lambda}}$ at roots of unity assuming (\ref{asm}). In particular, these findings provide a partial solution to the problem stated in \cite[Problem 2]{cw}. Throughout this chapter $\mathbb{K}$ is an algebraically closed field of arbitrary characteristics and all modules are the right modules.
\subsection*{Arrangement:} The paper is organized as follows. In Section $2$, we discuss some necessary facts for quantum Weyl algebra and the theory of Polynomial Identity algebras. Additionally, we present a key technique for calculating the PI degree of quantum affine space and the derivation erasing process, which remains independent of characteristics.
In Section $3$, we compute the explicit expression of PI degree for quantum Weyl algebra and some of its prime factors, under the assumptions (\ref{asm}). Moving on to Section 4, we present a classification of all maximal-dimensional simple modules over quantum Weyl algebra. Consequently, in Sections $5$ and $6$, we compute the center and the Azumaya locus for quantum Weyl algebra respectively. Finally, in Section 7, we extend our analysis to alternative quantum Weyl algebra, obtaining analogous results.
\section{{Preliminaries}}\label{pr}
In this section, we begin by revisiting essential results related to the quantum Weyl algebra. Subsequently, we recall relevant information about the Polynomial Identity algebra and the process of deleting derivation, enabling us to comment on the $\mathbb{K}$-dimension of simple modules over both quantized Weyl algebras at roots of unity. Throughout this paper, $\mathbb{K}$ will denote an algebraically closed field, and all modules under consideration shall be the right modules.
\subsection{Torsion and Torsionfree modules} 
Let $A$ be an algebra and $M$ be a right $A$-module and $S\subset A$ be a right Ore set. The submodule
\[\tors_{S}(M):=\{m\in M:ms=0\ \text{for some}\  s\in S\}\]
is called the $S$-torsion submodule of $M$. The module $M$ is said to be $S$-torsion if $\tors_{S}(M)=M$ and $S$-torsionfree if $\tors_{S}(M)=0$. If Ore set $S$ is generated by $x\in A$, we simply say that the $S$-torsion/torsionfree module $M$ is $x$-torsion/torsionfree.
\par A non zero element $x$ of an algebra ${A}$ is called a normal element if $x{A}={A}x$. Clearly, if $x$ is a normal element of $A$, then the set $\{x^i:~i\geq 0\}$ is an Ore set generated by $x$. The next lemma is obvious.
\begin{lemm}\label{itn}
Suppose that $A$ is an algebra, $x$ is a normal element of $A$ and $M$ is a simple $A$-module. Then either $Mx=0$ (if $M$ is $x$-torsion) or the map \[x_{M}:M\rightarrow M\  \text{given by}\  m\mapsto mx\] is an isomorphism (if $M$ is $x$-torsionfree).
\end{lemm}
The above lemma says that the action of a normal element on a simple module is either trivial or invertible.
\subsection{Quantum Weyl Algebras}
The quantum Weyl algebra ${{A}_n^{\underline{q},\Lambda}}$ has an iterated skew polynomial presentation with respect to the ordering of variables $y_1,x_1,y_2,x_2,\cdots,y_n,x_n$ of the form: 
$$\mathbb{K}[y_1][x_1,\tau_1,\delta_1][y_2,\sigma_2][x_2,\tau_2,\delta_2]\cdots [y_n,\sigma_n][x_n,\tau_n.\delta_n],$$
where the $\tau_j$ and $\sigma_{j}$ are $\mathbb{K}$-linear automorphisms and the $\delta_j$ are $\mathbb{K}$-linear $\tau_j$-derivations  such that
\begin{align*}
  \tau_j(y_i)&=q_i\lambda_{ij}y_i,\ i<j & \sigma_j(y_i)&=\lambda_{ij}^{-1}y_i,\ i<j\\
  \tau_j(x_i)&=q_i^{-1}\lambda_{ij}^{-1}x_i,\ i<j& \sigma_j(x_i)&=\lambda_{ij}x_i,\ i<j\\
  \tau_j(y_j)&=q_jy_j,\ \forall~j & \delta_j(x_i)&=\delta_j(y_i)=0,\ i<j\\
  \delta_j(y_j)&=1+\sum\limits_{i<j}(q_i-1)y_ix_i,\ \forall ~j&&
\end{align*}
This observation along with the skew Hilbert Basis Theorem (cf. \cite[Theorem 2.9]{mcr}) yields that the algebra ${A_n^{\underline{q},\Lambda}}$ is an affine noetherian domain. Moreover, the family of ordered monomials
\begin{equation}\label{kbasis}
    \{y_1^{a_1}x_1^{b_1}\cdots y_n^{a_n}x_n^{b_n}~:~a_i,b_i\in \mathbb{Z}_{\geq 0}\}
\end{equation} is a $\mathbb{K}$-basis of ${A_n^{\underline{q},\Lambda}}$. 
Let us define $z_0:=1$ and 
\[z_i:=x_iy_i-y_ix_i\ \ \text{for}\ \  1\leq i\leq n.\] 
These elements are going to play a crucial role throughout this chapter. It is easy to verify that for all $1\leq i\leq n$,
\[z_i=1+\sum_{j=1}^i(q_j-1)y_jx_j=q_i^{-1}\left(1+\sum_{j=1}^i(q_j-1)x_jy_j\right)=z_{i-1}+(q_i-1)y_ix_i.\]
Therefore the last of the listed defining relations for ${A_n^{\underline{q},\Lambda}}$ can be rewritten as
\[x_iy_i-q_iy_ix_i=z_{i-1},\ \ \forall \ \ 1\leq i\leq n.\] \begin{prop}\emph{(\cite{aj})}\label{l1}
Direct computations yield the following results for ${A_n^{\underline{q},\Lambda}}$.
\begin{enumerate}
\item For all $1\leq i\leq n$, $z_i$ is a normal element of ${A_n^{\underline{q},\Lambda}}$. More precisely, we have: 
\begin{itemize}
    \item[(a)]  For all $i,j$ with $1\leq i<j\leq n$, $z_ix_j=x_jz_i$ and $z_iy_j=y_jz_i$ 
    \item[(b)]  For all $i,j$ with $1\leq j\leq i\leq n$, $z_ix_j=q_j^{-1}x_jz_i$ and $z_iy_j=q_jy_jz_i$
    \item[(c)] For all $i,j$ with $1\leq i,j\leq n$, $z_iz_j=z_jz_i$.
\end{itemize}
\item The following identities hold in the algebra ${A_n^{\underline{q},\Lambda}}$, for $k\geq 1$:
\begin{itemize}
    \item[(a)] $x_i^ky_i=q_i^ky_ix_i^k+\left(1+q_i+\cdots+q_i^{k-1}\right)z_{i-1}x_i^{k-1}$,
    \item[(b)] $x_iy_i^k=q_i^ky_i^kx_i+\left(1+q_i+\cdots+q_i^{k-1}\right)z_{i-1}y_i^{k-1}$. 
\end{itemize}
\end{enumerate}
\end{prop}
\par In Section 7, we have presented all the necessary information and essential properties regarding the alternative quantum Weyl algebra ${\mathcal{A}_n^{\underline{q},\Lambda}}$. It is important to note that ${\mathcal{A}_n^{\underline{q},\Lambda}}$ and ${{A}_n^{\underline{q},\Lambda}}$ share a common localization, which plays a crucial role in studying ${\mathcal{A}_n^{\underline{q},\Lambda}}$.
\subsection{Polynomial Identity Algebras}
In this subsection, we recall some known facts concerning Polynomial Identity algebra that we shall be applying on (both) quantized Weyl algebras at roots of unity for further development. 
\par A ring $R$ is said to be a polynomial identity (PI) ring if $R$ satisfies some monic polynomial $f\in \mathbb{Z}\langle x_1,\cdots,x_k\rangle$ i.e., $f(r_1,\cdots,r_k)=0$ for all $r_i\in R$.  The minimal degree of a PI ring $R$ is the least degree of all monic polynomial identities for $R$. PI rings cover a large class of rings including commutative rings. Commutative rings satisfy the polynomial identity $x_1x_2-x_2x_1$ and therefore have minimal degree $2$. Let us recall a result which provides a sufficient condition for a ring to be PI. 
\begin{prop}\emph{(\cite[Corollary 13.1.13]{mcr})}{\label{f}}
If $R$ is a ring that is a finitely generated module over a commutative subring, then $R$ is a PI ring.
\end{prop}
Now in the root of unity setting, the above proposition yields the following result.
\begin{prop} \label{qwapia}
The quantum Weyl algebra ${A_n^{\underline{q},\Lambda}}$ is a PI algebra if and only if $q_i$'s and $\lambda_{ij}$'s are roots of unity.
\end{prop}
\begin{proof}
Suppose $q_i$ and $\lambda_{ij}$ are roots of unity such that $l=\lcmu(\ord(q_i),\ord(\lambda_{ij}):1\leq i,j\leq n)$. Let $Z$ be the subalgebra of ${A_n^{\underline{q},\Lambda}}$ generated by $x_i^{l},y_i^{l},\ \ \forall\ \ 1\leq i\leq n$. Then from the defining relations of ${A_n^{\underline{q},\Lambda}}$ along with Proposition \ref{l1}, the subalgebra $Z$ is central. Now from the $\mathbb{K}$-basis of ${A_n^{\underline{q},\Lambda}}$, one can easily verify that ${A_n^{\underline{q},\Lambda}}$ is  a finitely generated module over a central subalgebra $Z$. Hence it follows from Proposition \ref{f} that ${A_n^{\underline{q},\Lambda}}$ is a PI algebra.
\par For the converse, just note that the $\mathbb{K}$-subalgebra of ${A_n^{\underline{q},\Lambda}}$ generated by $x_i$ and $x_j$ with relation $x_ix_j=q_i\lambda_{ij}x_jx_i$ is not PI if $q_i$ or $\lambda_{ij}$ is not a root of unity (cf. \cite[Proposition I.14.2.]{brg}).
\end{proof}
We now define the PI degree of prime PI-algebras. This definition will suffice because the algebras covered in this article are all prime. As a consequence of the Artin-Wedderburn Theorem, any central simple algebra $A$ is isomorphic to a matrix ring over a central simple division ring. Hence $\dime_{Z(A)}(A)=n^2$ for some natural number $n$. From this, we define the PI degree of $A$ to be $n$. We now recall one of the fundamental results from the Polynomial Identity theory.
\begin{theo}[Posner’s Theorem {\cite[Theorem 13.6.5]{mcr}}]\label{pos}
 Let $A$ be a prime PI ring with centre $Z(A)$ and minimal
degree $d$. Let $S=Z(A)\setminus\{0\}$, $Q=AS^{-1}$ and $F=Z(A)S^{-1}$. Then $Q$ is a central simple algebra with centre $F$ and $\dime_{F}(Q)=(\frac{d}{2})^2$.
\end{theo} 
Note that the $Q$ in Posner’s theorem is PI and, since $Q$ is a central simple algebra, we can state its PI degree to be $\frac{d}{2}$ by the discussion above. Furthermore, as a result, \cite[I.13.2(6)]{brg}, $Q$ has the same minimal degree as $A$, namely $d$. Recognizing that the PI degree can be interpreted as some measure of how close to being commutative a PI-algebra is and that this, in turn, is related to its minimal degree, the definition of PI degree given above can be extended to all prime PI rings in the following way. 
\begin{defi}
The PI degree of a prime PI ring $A$ with minimal degree $d$ is $\pideg(A)=\frac{d}{2}$.
\end{defi}
\begin{rema}\label{piquotient}\normalfont
Based on this definition, for a prime PI ring $R$:
\begin{itemize}
    \item [(1)] $\pideg(S)\leq \pideg(R)$ for all subalgebra $S$ of $R$.
    \item [(2)] $\pideg(R/P)\leq \pideg(R)$ for all prime ideals $P\in \spect(R)$.     
\end{itemize}
\end{rema}
As a consequence of Posner's Theorem, every prime PI ring $R$ has a total ring of fractions $\cf(R)$ obtained by inverting all nonzero central elements of $R$. Now we obtain the following result \cite[Corollary I.13.3]{brg}: 
\begin{coro}\label{pisame}
Let $R$ be a prime PI ring. If $S$ is a subring of $\cf(R)$ with $R\subseteq S$, then $S$ is also a prime PI ring and $\pideg(S)=\pideg(R)$.
\end{coro}
Primitive PI ring exhibits a particularly nice structure, established in Kaplansky's Theorem (cf. \cite[ Theorem 13.3.8]{mcr}).  Now Kaplansky's Theorem has a striking consequence in the case of a prime affine PI algebra over an algebraically closed field. The following result provides an important link between the PI degree of a prime affine PI algebra over an algebraically closed field and the $\mathbb{K}$-dimension of its irreducible representations (cf. \cite[Theorem I.13.5, Lemma III.1.2]{brg}): 
\begin{prop}\label{sim}
Let $A$ be a prime affine PI algebra over an algebraically closed field $\mathbb{K}$, with PI-deg($A$) = $n$ and $V$ be a simple $A$-module. Then $V$ is a vector space over $\mathbb{K}$ of dimension $t$, where $t \leq n$, and $A/ann_A(V) \cong M_t(\mathbb{K})$. Moreover, the upper bound PI-deg($A$) is attained by some simple $A$-modules.
\end{prop}
\par  
\begin{rema}
In the roots of unity context, the quantum Weyl algebra is classified as a prime affine PI algebra. Consequently, according to Proposition \ref{sim}, it is quite clear that each simple ${A_n^{\underline{q},\Lambda}}$-module is finite dimensional and can have dimension at most $\pideg( {A_n^{\underline{q},\Lambda}}$). In Section \ref{lpi}, our focus will be on computing the PI degree of ${A_n^{\underline{q},\Lambda}}$. 
\end{rema} 
In the following two subsections, we will focus on deleting derivation process and PI degree of quantum affine spaces to compute the PI degree of the concerned algebras.
\subsection{Derivation Erasing Process} In this discussion we shall recall some results of PI degree parity between iterated Ore extension and associated quantum polynomial algebra depending on various assumptions on base ring, existing derivations, and base field. 
\par The main result of \cite{hh} states that if $R$ is a noetherian domain which is also an algebra over a field $\mathbb{K}$ then, under some quantum-like hypothesis, the iterated Ore extension $R_n=R[x_1,\sigma_1,\delta_1]\cdots[x_n,\sigma_n,\delta_n]$ and $T_n=R[x_1,\sigma_1]\cdots[x_n,\sigma_n]$ have the same PI degree.  This result was known earlier only in the case where $R$ is a prime algebra over a field of $\mathbb{K}$ of characteristic $0$ (cf. \cite{sj1}). In \cite{lm2}, A. Leroy and J. Matczuk presented a short and elementary proof of the derivation erasing process under the hypothesis that the iterated Ore extension is a prime PI ring. In particular, the following result was shown:

\begin{prop}\emph{(\cite[Theorem 7]{lm2})}\label{dr}
Suppose that $R=R_0$ is a prime PI algebra over a field $\mathbb{K}$ and $n\geq 1$. Let $R_i:=R_{i-1}[x_i,\sigma_i,\delta_i],~1\leq i\leq n$, be a sequence of Ore extensions such that each $\sigma_i$ is a $\mathbb{K}$-linear automorphism of $R_{i-1}$ and each $\delta_i$ is a $\mathbb{K}$-linear $\sigma_i$-derivation of $R_{i-1}$ such that:

\begin{enumerate}
    \item[(i)] $\sigma_i|_{R_0}$ is an automorphism of $R_0$ of finite order, for any $1\leq i \leq n$.
    \item[(ii)] $\sigma_i(x_j)=q_{ij}x_j$, where $q_{ij}\in \mathbb{K}^*$, for any $1\leq j<i\leq n$.
    \item[(iii)] $\delta_i$ is a $q_i$-skew $\sigma_i$-derivation of $R_{i-1}$, where $1\neq q_i\in \mathbb{K}$, for any $1\leq i \leq n$.
\end{enumerate}
Then the following conditions are equivalent:
\begin{enumerate}
    \item $R_n$ is a PI algebra.
    \item $T_n=R_0[y_1,\sigma'_1]\cdots [y_n,\sigma'_n]$ is a PI algebra, where $\sigma'_i|_{R_0}=\sigma_i|_{R_0}$ and $\sigma'_i(y_j)=q_{ij}y_j$, for $1\leq i \leq n$ and $1\leq j <i \leq n$;
    \item $q_{ij}$ is a root of unity, for any $1\leq j <i \leq n$;
    \item $\sigma_i$ is an automorphism of finite order of $R_{i-1}$, for any $1\leq i \leq n$.
\end{enumerate}
Moreover, if one of the above equivalent conditions holds, then the algebras $R_n$ and $T_n$ have isomorphic classical rings of quotients and equal PI degree.
\end{prop}
This erasing derivation process does not assume that the derivations are locally nilpotent or that the prime base ring is of zero characteristic. In this sense, it gives a generalization of the well-known Cauchon process of erasing derivations (cf. \cite{cau}). Both Theorem $1.2(1)$ and Corollary $(4.7)$ of \cite{hh} are direct consequences of the above proposition. 
\par Most of the quantum algebras can be presented as iterated Ore extension of the form $\mathbb{K}[x_1][x_2,\sigma_2,\delta_2]\cdots[x_n,\sigma_n,\delta_n]$, where the appearing automorphisms and skew derivations are as in the proposition (\ref{dr}) by taking $R_0=\mathbb{K},\sigma_1=id_{\mathbb{K}},\delta_1=0$. Thus such algebras have the same PI degree with quantum affine space $\mathbb{K}[x_1][x_2,\sigma_2]\cdots[x_n,\sigma_n]$. 
\subsection{PI Degree of Quantum Affine Spaces} 
Let $\mathbf{q}=(q_{ij})$ be a multiplicatively antisymmetric $(n\times n)$-matrix over $\mathbb{K}$, that is, $q_{ii}=1$ for all $i$ and $q_{ji}={q_{ij}}^{-1}$ for all $i\neq j$. Given such a matrix, the multiparameter quantum affine space of degree $n$ is the $\mathbb{K}$-algebra $\mathcal{O}_{\mathbf{q}}(\mathbb{K}^n)$ generated by the variables $x_1,\cdots ,x_n$ subject only to the relations
\begin{equation} \label{relation}
x_ix_j=q_{ij}x_jx_i, \ \ \ \forall\ \ \ 1 \leq i,j\leq n.
\end{equation}
It is of interest to know when a quantum affine space $\mathcal{O}_{\mathbf{q}}(\mathbb{K}^n)$ is a PI ring and what its PI degree is. Using Kaplansky's Theorem and Proposition \ref{f}, one can prove that $\mathcal{O}_{\mathbf{q}}(\mathbb{K}^n)$ is a PI ring if and only if all the $q_{ij}$ are roots of unity. The following result of De Concini and Procesi provides one of the key techniques for calculating the PI degree of a quantum affine space independent of characteristic.
\begin{prop}\emph{(\cite[Proposition 7.1]{di})}\label{quan}
Let $\mathbf{q}=\left(q_{ij}\right)$ be an  $n \times n$ multiplicatively antisymmetric matrix over $\mathbb{K}$.
 Suppose that $q_{ij}=q^{h_{ij}}$ for all $i,j$, where $q \in \mathbb{K}^*$ is a primitive $m$-th root of unity and the $h_{ij} \in \mathbb{Z}$. Let $h$ be the cardinality of the image of the homomorphism 
\[
    \mathbb{Z}^n \xrightarrow{(h_{ij})} \mathbb{Z}^n \xrightarrow{\pi} \left(\mathbb{Z}/m\mathbb{Z}\right)^n,
\]
where $\pi$ denotes the canonical epimorphism. Then \[\pideg(\mathcal{O}_{\mathbf{q}}(\mathbb{K}^n))=\pideg(\mathcal{O}_{\mathbf{q}}\left(\mathbb{(K^*)}^n\right)=\sqrt{h}.\]
\end{prop}
It is well known that a skew-symmetric matrix over $\mathbb{Z}$ such as our matrix $H:=(h_{ij})$ can be brought into a $2\times 2$ block diagonal form (commonly known as skew normal form) by a unimodular matrix $W\in GL_n(\mathbb{Z})$. Then $H$ is congruent to a block diagonal matrix of the form: 
\[WHW^{t}=\diagonal\left(\begin{pmatrix}
0&h_1\\
-h_1&0
\end{pmatrix},\cdots,\begin{pmatrix}
0&h_s\\
-h_s&0
\end{pmatrix},\bf{0}_{n-2s}\right),\]
where $\bf{0}_{n-2s}$ is the square matrix of zeros of dimension equals $\dime(\kere H)$, so that $2s=\ran(H)=n-\dime(\kere H)$ and $h_i\mid h_{i+1}\in \mathbb{Z}\setminus\{0\},~~\forall \ \ 1\leq i \leq s-1$. This nonzero $h_1,h_1,\cdots,h_s,h_s$ (each occurs twice) are called the invariant factors of $H$. The following result simplifies the calculation of $h$ in the statement of proposition (\ref{quan}) by the properties of the integral matrix $H$, namely the dimension of its kernel, along with its invariant factors and the value of $m$.
\begin{lemm}\emph{(\cite[Lemma 5.7]{ar})}\label{mainpi}
Take $1\neq q\in \mathbb{K}^*$, a primitive $m$-th root of unity. Let $H$ be a skew symmetric integral matrix associated to $\mathbf{q}$ with invariant factors $h_1,h_1,\cdots,h_s,h_s$. Then PI degree of $\mathcal{O}_{\mathbf{q}}(\mathbb{K}^n)$ is given as \[\pideg(\mathcal{O}_{\mathbf{q}}(\mathbb{K}^n))=\pideg(\mathcal{O}_{\mathbf{q}}\left(\mathbb{(K^*)}^n\right)=\prod_{i=1}^{\frac{n-\dime(\kere H)}{2}}\frac{m}{\gcdi(h_i,m)}.\]
\end{lemm}
Let us mention that Haynel applied Corollary $(4.7)$ of \cite{hh} and Proposition (\ref{quan}), a result of De Concini and Procesi, to compute PI degree explicitly of some quantum algebras such as the coordinate ring of odd/even-dimensional quantum Euclidean space, quantum symplectic space, coordinate ring of quantum matrices (uniparameter) and quantized Weyl algebras (uniparameter). \par In the next section, we will focus on computing the PI degree of multiparameter quantum Weyl algebras. 
\section{{PI degree for Quantum Weyl Algebras and its prime factors}}\label{lpi}
Let us first recall the assumption (\ref{asm}) on multiparameters. Hence the quantum Weyl algebra ${{A}_n^{\underline{q},\Lambda}}$ is a PI algebra. This section aims to calculate an explicit expression of PI degree for quantum Weyl algebras and some prime factor of it, with the assumption (\ref{asm}). Here we will use the derivation erasing process as in Proposition \ref{dr} and then a key technique for calculating the PI degree of quantum affine space as in Proposition \ref{quan}.
\subsection{PI degree for \texorpdfstring{${{A}_n^{\underline{q},\Lambda}}$}{TEXT}} Based on the iterated skew polynomial presentation of the quantum Weyl algebra ${{A}_n^{\underline{q},\Lambda}}$, we are now ready to use the derivation erasing independent of characteristic due to by A. Leroy and J. Matczuk \cite{lm2}.\\ 
\textbf{Step 1:} (Deleting Derivation) A straightforward calculation shows that the $\mathbb{K}$-linear maps $\tau_j$ and $\delta_j$ satisfy the skew relation $\delta_j\tau_j=q_j\tau_j\delta_j$ $(q_j\neq 1)$. Thus all the hypotheses of the derivation erasing process as in Proposition \ref{dr} are satisfied by the PI algebra ${{A}_n^{\underline{q},\Lambda}}$. Hence it follows that $\pideg {{A}_n^{\underline{q},\Lambda}}= \pideg \mathcal{O}_{B}(\mathbb{K}^{2n})$, where the $2n\times 2n$ matrix of relations $B$ is comprised of $2\times 2$ blocks:
\[B_{ii}=
\begin{pmatrix} 
1 & q_i^{-1} \\
q_i & 1
\end{pmatrix},\ \
B_{ij}=\begin{pmatrix}
\lambda_{ij}&q_i^{-1}\lambda_{ij}^{-1}\\
\lambda_{ij}^{-1}&q_i\lambda_{ij}
\end{pmatrix}\ i<j,\ \
B_{ji}=\begin{pmatrix}
\lambda_{ij}^{-1}&\lambda_{ij}\\
q_i\lambda_{ij}&q_i^{-1}\lambda_{ij}^{-1}
\end{pmatrix}\ i<j.
\]
\textbf{Step 2:} (Integral Matrix) Now we form an integral matrix associated with this $B$ under the assumption (\ref{asm}). Then the multiplicative group $\Gamma$ generated by all the multiparameters $q_i$ and $\lambda_{ij}$ is a cyclic group of order $l_n$ and suppose $q$ be a generator of the group $\Lambda$. Note that we have a chain of cyclic subgroups \[\langle q_1\rangle\subseteq\langle q_2\rangle\subseteq\cdots\subseteq\langle q_n\rangle=\langle q\rangle\] due to the divisibility condition $l_1|\cdots|l_n$. Therefore we can choose $a_i:=\frac{l_{i+1}}{l_i}$ such that $\langle q_i\rangle=\langle q_{i+1}^{a_i}\rangle$ for all $1\leq i\leq n-1$. Set $a_n=1$. So there exist integers $b_i,b_{ij}$ such that \[q_i=q_{i+1}^{a_ib_i},\ q_n=q^{a_nb_n}\  \text{and} \ \lambda_{ij}=q_i^{b_{ij}}=q_{i+1}^{a_ib_ib_{ij}}\ \ \forall\  1\leq i\leq j\leq n.\] Note that $\gcdi(b_i,l_i)=1$ for all $1\leq i\leq n$. Define 
\[s_i:=a_i\cdots a_n,\ k_i:=b_i\cdots b_n,\ k_{ij}:=b_{ij}\ \ \forall \ 1\leq i\leq j\leq n.\] Thus we can write
\[q_i=q^{s_ik_i}\  \text{and} \ \lambda_{ij}=q_i^{k_{ij}}=q^{s_ik_ik_{ij}}\ \ \forall \ 1\leq i\leq j\leq n.\]
Now the powers of this $q$ from the matrix $B$ give a $2n\times 2n$ integer matrix $B^{'}$ comprised of $2\times 2$  blocks
\begin{align*}
&B^{'}_{ii}=
\begin{pmatrix} 
0 & -s_ik_i \\
s_ik_i & 0
\end{pmatrix},\ \
B^{'}_{ij}=
\begin{pmatrix}
s_ik_ik_{ij}&-s_ik_i-s_ik_ik_{ij}\\
-s_ik_ik_{ij}&s_ik_i+s_ik_ik_{ij}
\end{pmatrix}\ i<j,\\
&B^{'}_{ji}=
\begin{pmatrix}
-s_ik_ik_{ij}& s_ik_ik_{ij}\\
s_ik_i+s_ik_ik_{ij}&-s_ik_i-s_ik_ik_{ij}
\end{pmatrix}\ i<j.
\end{align*}
\textbf{Step 3:} (Final Step) The PI degree of ${{A}_n^{\underline{q},\Lambda}}$ can be calculated using the integral matrix $B'$ from Proposition \ref{quan}. The cardinality of the image of the homomorphism 
\begin{equation}\label{pihom}
   \mathbb{Z}^{2n} \xrightarrow{B'} \mathbb{Z}^{2n} \xrightarrow{\pi} \left(\mathbb{Z}/l_n\mathbb{Z}\right)^{2n}, 
\end{equation} determined by the integral matrix $B'$ remains unchanged if we perform certain elementary reductions on $B'$ beforehand. Let us apply the following sequence of elementary operations on $B^{'}$ to reduce it into a simpler form:
\begin{itemize}
    \item For $i=2,4,\cdots, 2n$, replace $\rowa(i)$ with $\rowa(i) + \rowa(i-1)$.
    \item For $i=2,4,\cdots, 2n$, replace $\colu(i)$ with $\colu(i) + \colu(i-1)$.
    \item For $1\leq i < j \leq n$, replace $\rowa(2j-1)$ by $\rowa(2j-1)+k_{ij}\rowa(2i)$ and replace $\rowa(2j)$ with $\rowa(2j)-\rowa(2i)$.
    \item For $1\leq i < j \leq n$, replace $\colu(2j-1)$ by $\colu(2j-1)+k_{ij}\colu(2i)$ and replace $\colu(2j)$ with $\colu(2j)-\colu(2i)$.
\end{itemize}
Then we have a $2n\times 2n$ integral matrix $B^{''}$ of this form 
$$B^{''}=\diagonal\left( \begin{pmatrix} 
0 & -s_1k_1 \\
s_1k_1 & 0
\end{pmatrix},\cdots,\begin{pmatrix} 
0 & -s_nk_n \\
s_nk_n & 0
\end{pmatrix}\right).$$
Therefore the cardinally of the image of homomorphism (\ref{pihom}) induced by this integral matrix $B''$ is given by $\displaystyle\prod\limits_{i=1}^{n} [\ord(s_ik_i)]^2$, where order is taken in the additive group $\mathbb{Z}/{l_n\mathbb{Z}}$.
Thus we have 
\[\pideg {A}_n^{\underline{q},\Lambda}=\displaystyle\prod\limits_{i=1}^{n} \ord(s_ik_i)=\displaystyle\prod\limits_{i=1}^{n} \frac{l_n}{\gcdi(s_ik_i,l_n)}=\displaystyle\prod\limits_{i=1}^{n} \ord (q_i)=\displaystyle\prod\limits_{i=1}^{n} l_i.\]
Finally, we have proved the following
\begin{theo}\label{pitheorem}
The PI degree of ${A}_n^{\underline{q},\Lambda}$ is given by\[\pideg{{A}_n^{\underline{q},\Lambda}}=\prod\limits_{i=1}^nl_i,\] under the roots of unity assumption (\ref{asm}) on the defining multiparameters.  
\end{theo}
\begin{rema}\normalfont
In view of Proposition \ref{sim}, it is worth noting that, there exists a simple ${A}_n^{\underline{q},\Lambda}$-module with a $\mathbb{K}$-dimension of $\prod\limits_{i=1}^nl_i$. The subsequent Section \ref{intqwa} will provide an explicit construction and classification of these simple modules.
\end{rema}
\subsection{PI degree for Prime Factors of \texorpdfstring{${{A}_n^{\underline{q},\Lambda}}$}{TEXT}}\label{sr} Fix an index $r$ with $1\leq r\leq n$. Then the ideal generated by the normal element $z_r:=x_ry_r-y_rx_r$ is a completely prime ideal of ${{A}_n^{\underline{q},\Lambda}}$ (see \cite[Proposition 4.5]{aj}) and hence the factor algebra ${{A}_n^{\underline{q},\Lambda}}/\langle z_r \rangle$ is a domain. Therefore the algebra ${{A}_n^{\underline{q},\Lambda}}/\langle z_r \rangle$ is a prime affine PI algebra under the root of unity assumption (\ref{asm}). Here, our objective is to explicitly determine the PI degree for the factor algebra ${{A}_n^{\underline{q},\Lambda}}/\langle z_r \rangle$.\\
\noindent\textbf{Step 1:} Let $R_r$ denote the $\mathbb{K}$-algebra generated by the variables $X_1,Y_1,\cdots,X_n,Y_n$ subject to the relations:
\begin{align}\label{qwadr1}
Y_iY_j&=\lambda_{ij}Y_jY_i \ \ \ \ \ \forall \ \ \ 1\leq i<j\leq n\nonumber\\
X_iX_j&=q_i\lambda_{ij}X_jX_i \ \ \ \ \ \forall \ \ \ 1\leq i<j\leq n\nonumber\\
X_iY_j&=\lambda_{ij}^{-1}Y_jX_i\ \ \ \ \ \forall \ \ \ 1\leq i<j\leq n\\
Y_iX_j&=q_i^{-1}\lambda_{ij}^{-1}X_jY_i \ \ \ \ \ \forall \ \ \ 1\leq i<j\leq n\nonumber\\
X_rY_r&=Y_rX_r\nonumber\\
X_iY_i-q_iY_iX_i&=1+\sum_{k=1}^{i-1}(q_k-1)Y_kX_k \ \ \ \ \ \forall \ \ \ 1\leq i\neq r\leq n\nonumber
\end{align}
The algebra $R_r$ is similar to quantum Weyl algebra except for the defining relation between $x_r$ and $y_r$. The algebra $R_r$ is also a PI algebra in the root of unity setting. The following result establishes a connection between the algebras ${{A}_n^{\underline{q},\Lambda}}$ and $R_{r}$.
\begin{lemm}\label{cpfqwa}
The factor algebra ${{A}_n^{\underline{q},\Lambda}}/\langle z_r \rangle$ is isomorphic to the factor algebra $R_{r}/I_{r}$, where $I_r$ is the ideal of $R_r$ generated by $1+\sum\limits_{i=1}^{r}(q_i-1)Y_iX_i$. 
\end{lemm}
\begin{proof}
Define a $\mathbb{K}$-linear map $\Phi:{{A}_n^{\underline{q},\Lambda}}\longrightarrow R_{r}/I_{r}$ as follows: \[x_i\mapsto X_i+I_r,\ y_i\mapsto Y_i+I_{r}.\] We can easily verify that $\Phi$ is an algebra homomorphism with $\langle z_r \rangle\subseteq \kere(\Phi)$. Therefore by the universal property of quotient algebras, there is a unique $\mathbb{K}$-algebra homomorphism $\overline{\Phi}:{{A}_n^{\underline{q},\Lambda}}/\langle z_r \rangle\longrightarrow R_{r}/I_{r}$ such that 
\[\overline{\Phi}(x_i+\langle z_r \rangle)=X_i+I_r,\ \ \overline{\Phi}(y_i+\langle z_r \rangle)=Y_i+I_r\ \ \forall\ \ 1\leq i\leq n.\]
\par On the other hand, define a $\mathbb{K}$-linear map $\Psi:R_{r}\longrightarrow {{A}_n^{\underline{q},\Lambda}}/\langle z_r \rangle$ as follows:
\[X_i\mapsto x_i+\langle z_r \rangle,\ Y_i\mapsto y_i+\langle z_r \rangle.\] We can easily check that $\Psi$ is an algebra homomorphism such that $I_r\subseteq \kere(\Psi)$. Thus by the universal property of quotient algebras, we have a unique $\mathbb{K}$-algebra homomorphism $\overline{\Psi}:R_{r}/I_{r}\longrightarrow {{A}_n^{\underline{q},\Lambda}}/\langle z_r \rangle$ such that 
\[\overline{\Psi}(X_i+I_r)=x_i+\langle z_r \rangle,\ \ \overline{\Psi}(Y_i+I_r)=y_i+\langle z_r \rangle\ \ \forall\ \ 1\leq i\leq n.\]
Therefore $\overline{\Phi}$ and $\overline{\Psi}$ are inverse to each other and the lemma follows.
\end{proof}
Now based on Lemma \ref{cpfqwa} and Remark \ref{piquotient}, we can conclude that 
\begin{equation}\label{noteq}
    \pideg ({{A}_n^{\underline{q},\Lambda}}/\langle z_r \rangle)=\pideg{(R_r/I_r)}\leq \pideg{(R_r)}.
\end{equation} 
\noindent\textbf{Step 2:} Our focus now shifts to computing the PI degree for the algebra $R_r$. Similar to the quantum Weyl algebra, the algebra $R_r$ has an iterated skew polynomial presentation with respect to the ordering of variables $Y_1,X_1,\cdots,Y_n,X_n$ of the form: 
\begin{equation}\label{oc}
  \mathbb{K}[Y_1][X_1,\tau_1,\delta_1][Y_2,\sigma_2][X_2,\tau_2,\delta_2]\cdots [Y_n,\sigma_n][X_n,\tau_n.\delta_n]  
\end{equation}
where the $\tau_j$ and $\sigma_{j}$ are $\mathbb{K}$-linear automorphisms and the $\delta_j$ are $\mathbb{K}$-linear $\tau_j$-derivations  such that
\begin{align*}
 \tau_j(Y_i)&=q_i\lambda_{ij}Y_i,\ i<j  & \sigma_j(Y_i)&=\lambda_{ij}^{-1}Y_i,\ i<j\\
 \tau_j(X_i)&=q_i^{-1}\lambda_{ij}^{-1}X_i,\ i<j & \sigma_j(X_i)&=\lambda_{ij}X_i,\ i<j\\
 \tau_j(Y_j)&=q_jY_j,\ \forall~j\neq r & \tau_r(Y_r)&=Y_r\\
  \delta_j(Y_j)&=1+\sum\limits_{i<j}(q_i-1)Y_iX_i,\ \forall ~j\neq r& \delta_r(Y_r)&=0
\end{align*}
We can easily verify that the $\mathbb{K}$-linear maps $\tau_j$ and $\delta_j$ satisfy the skew relation $\delta_j\tau_j=q_j\tau_j\delta_j$ $(q_j\neq 1)$ from this iterated skew polynomial presentation of $R_r$. Hence, all the hypotheses of the derivation erasing process as in \cite[Theorem 7.1]{lm2} are satisfied by the PI algebra $R_r$. Therefore we obtain $\pideg (R_r)= \pideg \mathcal{O}_{C}(\mathbb{K}^{2n})$, where the $2n\times 2n$ matrix of relations $C$ is comprised of $2\times 2$ blocks 
\begin{align*}
C_{ii}&=
\begin{pmatrix} 
1 & q_i^{-1} \\
q_i & 1
\end{pmatrix}\ i\neq r &
C_{ij}&=\begin{pmatrix}
\lambda_{ij}&q_i^{-1}\lambda_{ij}^{-1}\\
\lambda_{ij}^{-1}&q_i\lambda_{ij}
\end{pmatrix}\ i<j \\
C_{rr}&=\begin{pmatrix} 
1 & 1 \\
1 & 1
\end{pmatrix}&
C_{ji}&=\begin{pmatrix}
\lambda_{ij}^{-1}&\lambda_{ij}\\
q_i\lambda_{ij}&q_i^{-1}\lambda_{ij}^{-1}
\end{pmatrix}\ i<j.    
\end{align*}
Now as in the step 2 of Section \ref{lpi}, the integral matrix $C'$ associated with $C$ under the assumption (\ref{asm}) is comprised of $2\times 2$ blocks
\begin{align*}
C^{'}_{ii}&=
\begin{pmatrix} 
0 & -s_ik_i \\
s_ik_i & 0
\end{pmatrix}\ i\neq r &  
C^{'}_{ij}&=
\begin{pmatrix}
s_ik_ik_{ij}&-s_ik_i-s_ik_ik_{ij}\\
-s_ik_ik_{ij}&s_ik_i+s_ik_ik_{ij}
\end{pmatrix}\ i<j\\
C^{'}_{rr}&=
\begin{pmatrix} 
0 & 0 \\
0 & 0
\end{pmatrix}&
C^{'}_{ji}&=
\begin{pmatrix}
-s_ik_ik_{ij}& s_ik_ik_{ij}\\
s_ik_i+s_ik_ik_{ij}&-s_ik_i-s_ik_ik_{ij}
\end{pmatrix}\ i<j.
\end{align*}
The $\pideg R_r$ can be computed using the integral matrix $C'$ as stated in Proposition \ref{quan}. However, we can perform certain elementary reductions on $C'$ before using it to compute the cardinality of the image of the homomorphism 
\begin{equation}\label{piqwa1}
   \mathbb{Z}^{2n} \xrightarrow{C'} \mathbb{Z}^{2n} \xrightarrow{\pi} \left(\mathbb{Z}/l_n\mathbb{Z}\right)^{2n}, 
\end{equation} and this will not change the result. We note that our choice of $s_i$ and $k_i$ satisfies $s_nk_n|\cdots|s_1k_1$, which we will use in the following operations. To simplify $C'$, we will perform the following elementary operations: 
\begin{itemize}
    \item For $i=2,4,\cdots, 2n$, replace $\rowa(i)$ with $\rowa(i) - \rowa(i-1)$.
    \item For $i=2,4,\cdots, 2n$, replace $\colu(i)$ with $\colu(i) - \colu(i-1)$.
    \item For $1\leq i < j \leq n$ and $i<r$, replace $\rowa(2j-1)$ by $\rowa(2j-1)+k_{ij}\rowa(2i)$ and replace $\rowa(2j)$ with $\rowa(2j)-\rowa(2i)$.
    \item For $1\leq i < j \leq n$ and $i<r$, replace $\colu(2j-1)$ by $\colu(2j-1)+k_{ij}\colu(2i)$ and replace $\colu(2j)$ with $\colu(2j)-\colu(2i)$.
    \item For $1\leq i < j \leq n$ and $r<i$, replace $\rowa(2j)$ by $\rowa(2j)-\rowa(2i)$ and replace $\colu(2j)$ with $\colu(2j)-\colu(2i)$.
    \item For $r+2\leq j\leq n$ and $r\leq i<j$, replace $\rowa(2i-1)$ with \[\rowa(2i-1)-\displaystyle\frac{s_ik_ik_{ij}}{s_jk_j}\rowa(2j)\] and replace $\colu(2i-1)$ with \[\colu(2i-1)-\displaystyle\frac{s_ik_ik_{ij}}{s_jk_j}\colu(2j).\]
    \item Replace $\rowa(2r-1)$ by \[\rowa(2r-1)-\displaystyle\frac{s_rk_rk_{r(r+1)}}{s_{r+1}k_{r+1}}\rowa(2r+2)\] and replace $\colu(2r-1)$ by \[\colu(2r-1)-\displaystyle\frac{s_rk_rk_{r(r+1)}}{s_{r+1}k_{r+1}}\colu(2r+2),\] then replace $\rowa(2r-1)$ with \[\rowa(2r-1)-\displaystyle\frac{s_rk_r}{s_{r+1}k_{r+1}}\rowa(2r+1)\] and replace $\colu(2r-1)$ by \[\colu(2r-1)-\displaystyle\frac{s_rk_r}{s_{r+1}k_{r+1}}\colu(2r+1).\]
\end{itemize}
Finally, we obtain a $2n\times 2n$ integer matrix $C^{''}$ with the following diagonal form 
$$C^{''}=\diagonal\left( \begin{pmatrix} 
0 & -s_1k_1 \\
s_1k_1 & 0
\end{pmatrix},\cdots,\begin{pmatrix} 
0 & 0 \\
0 & 0
\end{pmatrix},\cdots,\begin{pmatrix} 
0 & -s_nk_n \\
s_nk_n & 0
\end{pmatrix}\right).$$
Then the cardinality of the image of the homomorphism (\ref{piqwa1}), induced by the matrix $C''$, can be expressed as $\displaystyle\prod\limits_{\substack{i=1\\i\neq r}}^{n} \left(\ord(s_ik_i)\right)^2$, where order is taken in the additive group $\mathbb{Z}/{l_n\mathbb{Z}}$.
Consequently, based on the aforementioned discussions and Proposition \ref{quan}, we can deduce the following:
\[\pideg (R_r)=\displaystyle\prod\limits_{\substack{i=1\\i\neq r}}^{n} \ord(s_ik_i)=\displaystyle\prod\limits_{\substack{i=1\\i\neq r}}^{n} \frac{l_n}{\gcdi(s_ik_i,l_n)}=\displaystyle\prod\limits_{\substack{i=1\\i\neq r}}^{n} \ord (q_i)=\displaystyle\prod\limits_{\substack{i=1\\i\neq r}}^{n} l_i.\] 
Thus from (\ref{noteq}), we can conclude that 
\begin{equation}\label{ineq}
\pideg ({{A}_n^{\underline{q},\Lambda}}/\langle z_r \rangle)\leq \displaystyle\prod\limits_{\substack{i=1\\i\neq r}}^{n} l_i.
\end{equation}
\textbf{Step 3:} Finally we aim to establish the equality in (\ref{ineq}). For simplicity let \[X_i:=x_i+\langle z_r \rangle\ \ \text{and}\ \  Y_i:=y_i+\langle z_r \rangle\] denote the cosets of $x_i$ and $y_i$ in ${{A}_n^{\underline{q},\Lambda}}/\langle z_r \rangle$ respectively. Suppose $A(r)$ be the subalgebra of the factor algebra ${{A}_n^{\underline{q},\Lambda}}/\langle z_r \rangle$ generated by $X_i,Y_i$ for all $i\neq r$. Clearly $A(r)$ is the $\mathbb{K}$-algebra generated by the $(2n-2)$-variables $X_i,Y_i$ for all $1\leq i\neq r\leq n$ and the following relations:
\begin{align}
Y_iY_j&=\lambda_{ij}Y_jY_i \ \ \ \ \ \forall \ \ \ 1\leq i<j\leq n\nonumber\\
X_iX_j&=q_i\lambda_{ij}X_jX_i \ \ \ \ \ \forall \ \ \ 1\leq i<j\leq n\nonumber\\
X_iY_j&=\lambda_{ij}^{-1}Y_jX_i\ \ \ \ \ \forall \ \ \ 1\leq i<j\leq n\\
Y_iX_j&=q_i^{-1}\lambda_{ij}^{-1}X_jY_i \ \ \ \ \ \forall \ \ \ 1\leq i<j\leq n\nonumber\\
X_iY_i-q_iY_iX_i&=1+\sum_{k=1}^{i-1}(q_k-1)Y_kX_k \ \ \ \ \ \forall \ \ \ 1\leq i\leq r-1\nonumber\\
X_{r+1}Y_{r+1}&=q_{r+1}Y_{r+1}X_{r+1}\nonumber\\
X_iY_i-q_iY_iX_i&=\sum_{k=r+1}^{i-1}(q_k-1)Y_kX_k \ \ \ \ \ \forall \ \ \ r+2\leq i\leq n\nonumber
\end{align}
As a result, based on Remark \ref{piquotient}, we have: 
\begin{equation}\label{subfactor1}
    \pideg({A}(r))\leq \pideg({{A}_n^{\underline{q},\Lambda}}/\langle z_r \rangle)
\end{equation}
Moreover, applying the same reasoning as discussed in Step 2 of Subsection \ref{sr}, we can obtain that
\begin{equation}\label{subfactor2}
    \pideg({{A}(r)})=\prod\limits_{\substack{i=1\\ i\neq r}}^{n}l_i
\end{equation} Therefore by combining the relations (\ref{ineq}), (\ref{subfactor1}) and (\ref{subfactor2}), we can establish the following result:
\begin{theo}\label{piin}
Assuming the root of unity assumption (\ref{asm}) on the defining multiparameters, the PI degree of a prime factor of ${A}_n^{\underline{q},\Lambda}$ by the normal element $z_r$ is given by $\displaystyle\prod\limits_{\substack{i=1\\i\neq r}}^{n} l_i$.
\end{theo}
\noindent\textbf{Question:} As the factor algebra ${{A}_n^{\underline{q},\Lambda}}/\langle z_r\rangle$ is classified as a prime affine PI algebra, therefore by Proposition \ref{sim}, there exists a $z_r$-torsion simple ${{A}_n^{\underline{q},\Lambda}}$-module with $\mathbb{K}$-dimension $\prod\limits_{\substack{i=1\\i\neq r}}^{n} l_i$ for each $1\leq r\leq n$. It would be interesting to explore the construction of such simple modules and subsequently classify them. So let us give it a try!
\section{Maximal Dimensional Simple \texorpdfstring{${{A}_n^{\underline{q},\Lambda}}$}{TEXT}-modules}\label{intqwa} In view of Theorem \ref{piin}, the following result provides a necessary condition for an ${A}_n^{\underline{q},\Lambda}$-module to attain maximal $\mathbb{K}$-dimension.
\begin{theo}\label{onepart}
Let $M$ be a maximal $\mathbb{K}$-dimensional (i.e., equal to the $\pideg {{A}_n^{\underline{q},\Lambda}}$) simple ${{A}_n^{\underline{q},\Lambda}}$-module. Then $M$ is $z_i$-torsionfree for all $1\leq i\leq n$.
\end{theo} 
\begin{proof}
Suppose $M$ be a $z_r$-torsion simple $ {{A}_n^{\underline{q},\Lambda}}$-module for some $1\leq r\leq n$. As $z_r$ is normal element, $M$ becomes a simple module over the factor algebra ${{A}_n^{\underline{q},\Lambda}}/\langle z_r\rangle$. Note that the algebra ${{A}_n^{\underline{q},\Lambda}}/\langle z_r\rangle$ is prime affine PI algebra. Hence by Proposition \ref{sim} we have 
\[\dime_{\mathbb{K}}(M)\leq \pideg ({{A}_n^{\underline{q},\Lambda}}/\langle z_r\rangle).\]
Now by Theorem \ref{piin}, it follows that \[\dime_{\mathbb{K}}(M)\leq \pideg ({{A}_n^{\underline{q},\Lambda}}/\langle z_r\rangle)<\pideg ({{A}_n^{\underline{q},\Lambda}}).\] Thus we arrive at a contradiction. Hence we conclude that the maximal $\mathbb{K}$-dimensional simple ${{A}_n^{\underline{q},\Lambda}}$-modules are $z_i$-torsionfree for all $1\leq i\leq n$.
\end{proof} 
In the following subsections, we aim to prove the converse of the above theorem, by classifying all $z_i$-torsionfree simple ${{A}_n^{\underline{q},\Lambda}}$-modules. Denote $[1,n]:=\{1,2,\cdots,n\}$ for $n\in \mathbb{N}$.
\subsection{Construction of Simple Modules:}
For any two subsets $I$ and $J$ of indices $[1,n]$, let $\underline{\mu}(I,J):=(\mu_1,\cdots,\mu_n)\in \mathbb{K}^n$ such that $\mu_i=0$ iff $i\in I\cap J$ and let $\underline{\gamma}(I):=(\gamma_1,\cdots,\gamma_n)\in \mathbb{({K}^*)}^{n}$ such that $q_i\gamma_i=\gamma_{i-1}$ if $i\in I$. Given such $\underline{\mu}(I,J)$ and $\underline{\gamma}(I)$, let $M\left(\underline{\mu}(I,J),\underline{\gamma}(I)\right)$ be the $\mathbb{K}$-vector space with basis consisting of \[\{e(\underline{a})~|~\underline{a}=(a_1,\cdots,a_n),~0\leq a_i \leq l_i-1,~ 1\leq i\leq n\}.\]
Let $\epsilon_i(\pm 1)$ denote the vector $(0,\cdots,{\pm 1},\cdots,0)$ whose $i$-th coordinate is $\pm 1$ and the remaining are zero. Define the ${{A}_n^{\underline{q},\Lambda}}$-module structure on the $\mathbb{K}$-space $M\left(\underline{\mu}(I,J),\underline{\gamma}(I)\right)$ by the action of each generator on the basis vectors as follows: for $1\leq i \leq n$,
\begin{align*}
e(\underline{a})x_i&=
\begin{cases}
    \mu_i\mathbb{A}(\underline{a};i)e(\underline{a}\oplus\epsilon_i(+1)),&\ i\in [1,n]\setminus I\\
    \mu_i^{-1}\mathbb{A}(\underline{a};i)\displaystyle\frac{(q_i^{-a_i}-1)\gamma_{i-1}}{q_i-1}~e(\underline{a}\oplus\epsilon_i(-1)),&\ i\in I\setminus (I\cap J)\\
    \mathbb{A}(\underline{a};i)\displaystyle\frac{(q_i^{-a_i}-1)\gamma_{i-1}}{q_i-1}~e(\underline{a}\oplus\epsilon_i(-1)),&\ a_i\neq 0,\ i\in I\cap J\\
      0,& a_i=0,~i\in I\cap J
\end{cases}\\
e(\underline{a})y_i&=
\begin{cases}
     \mu_i^{-1} \mathbb{B}(\underline{a};i)\displaystyle\frac{q^{a_i}_i\gamma_i-\gamma_{i-1}}{q_i-1}~e(\underline{a}\oplus\epsilon_i(-1)),&\ i\in [1,n]\setminus I\\
     \mu_i \mathbb{B}(\underline{a};i)~e(\underline{a}\oplus\epsilon_i(+1)),& \ i\in I\setminus (I\cap J)\\
     \mathbb{B}(\underline{a};i)~e(\underline{a}\oplus\epsilon_i(+1)),& a_i\neq l_i-1,\ i\in I\cap J\\
      0,& a_i=l_i-1,~i\in I\cap J
\end{cases}
\end{align*}
where $\gamma_0=1$ and $\oplus$ is addition in the additive group $\bigoplus_{i=1}^{n}\mathbb{Z}/l_i\mathbb{Z}$ and the scalars $\mathbb{A}(\underline{a};i)$ and $\mathbb{B}(\underline{a};i)$ are given by
\begin{align*}
\mathbb{A}(\underline{a};i)&=\begin{cases}
1, &i=1\\
 \displaystyle\left(\prod_{\substack {s<i \\ s \notin I}}
    (q_s\lambda_{si})^{a_s}\right)\left(\prod_{\substack {s<i \\ s \in I}}(q_s\lambda_{si})^{-a_s}\right),& 2\leq i\leq n
\end{cases}\\
\mathbb{B}(\underline{a};i)&=\begin{cases}
1,&i=1\\
\displaystyle\left(\prod_{\substack {s<i \\ s \notin I}}
    \lambda_{si}^{-a_s}\right)\left(\prod_{\substack {s<i \\ s \in I}}\lambda_{si}^{a_s}\right),&2\leq i\leq n
    \end{cases}
\end{align*}
In order to establish the well-definedness we need to check that the $\mathbb{K}$-endomorphisms of $M\left(\underline{\mu}(I,J),\underline{\gamma}(I)\right)$ defined by the above rules satisfy the defining relations of ${{A}_n^{\underline{q},\Lambda}}$. Indeed we can compute 
\begin{align*}
    e(\underline{a})x_iy_i&=\begin{cases}
    \mathbb{A}(\underline{a};i)\mathbb{B}(\underline{a}\oplus\epsilon_i(+1);i)\displaystyle\frac{q_i^{a_i+1}\gamma_i-\gamma_{i-1}}{q_i-1}e(\underline{a}),&i\in [1,n]\setminus I\\
     \mathbb{A}(\underline{a};i)\mathbb{B}(\underline{a}\oplus\epsilon_i(-1);i)\displaystyle\frac{(q_i^{-a_i}-1)\gamma_{i-1}}{q_i-1}e(\underline{a}),&a_i\neq 0, i\in I\\
     0,&a_i=0,i\in I
    \end{cases}\\
e(\underline{a})y_ix_i&=\begin{cases}
 \mathbb{A}(\underline{a}\oplus\epsilon_i(-1);i)\mathbb{B}(\underline{a};i)\displaystyle\frac{q_i^{a_i}\gamma_i-\gamma_{i-1}}{q_i-1}e(\underline{a}),&i\in [1,n]\setminus I\\
  \mathbb{A}(\underline{a}\oplus\epsilon_i(+1);i)\mathbb{B}(\underline{a};i)\displaystyle\frac{(q_i^{-(a_i+1)}-1)\gamma_{i-1}}{q_i-1}e(\underline{a}),& a_i\neq l_i-1,i\in I\\
  0,&a_i=l_i-1,i\in I\\
\end{cases}
\end{align*}
so that for each $i\in [1,n]$: \[e(\underline{a})(x_iy_i-q_iy_ix_i)=\displaystyle\left(\prod_{\substack {s<i \\ s \notin I}}q_s^{a_s}\right)\left(\prod_{\substack {s<i \\ s \in I}}q_s^{-a_s}\right)\gamma_{i-1}e(\underline{a})\] and 
\[e(\underline{a})\left(1+\sum_{k=1}^{i-1}(q_k-1)y_kx_k\right)=\displaystyle\left(\prod_{\substack {s<i \\ s \notin I}}q_s^{a_s}\right)\left(\prod_{\substack {s<i \\ s \in I}}q_s^{-a_s}\right)\gamma_{i-1}e(\underline{a}).\] The remaining relations can be easily verified. Therefore the above action on the $\mathbb{K}$-space $M\left(\underline{\mu}(I,J),\underline{\gamma}(I)\right)$ is well defined. Now we have the following result:
\begin{theo}
The module $M\left(\underline{\mu}(I,J),\underline{\gamma}(I)\right)$ is a simple ${{A}_n^{\underline{q},\Lambda}}$-module of dimension $\prod\limits_{i=1}^n l_i$.
\end{theo}
\begin{proof}
Let $P$ be a nonzero submodule of $M\left(\underline{\mu}(I,J),\underline{\gamma}(I)\right)$. We claim that $P$ contains a basis vector of the form $e(a_1,\cdots,a_n)$. Indeed, any member $p\in P$ is a finite $\mathbb{K}$-linear combination of such vectors. i.e.,
\[
  p:=\sum_{\text{finite}} \lambda_k~e\left(a_1^{(k)},\cdots,a_n^{(k)}\right)  
\]
for some $\lambda_k\in \mathbb{K}$. Suppose there exist two non-zero coefficients, say, $\lambda_u,\lambda_v$.
We can choose the smallest index $1\leq r\leq n$ such that $a_r^{(u)}\neq a_r^{(v)}$ in ${\mathbb{Z}}/{l_r\mathbb{Z}}$. Then the vectors $e\left(a_1^{(u)},\cdots,a_n^{(u)}\right)$ and $e\left(a_1^{(v)},\cdots,a_n^{(v)}\right)$ are eigenvectors of $z_r$ associated with the eigenvalues \[\Lambda_u=
    \displaystyle\left(\prod_{\substack {s<r \\ s \notin I}}
    q_s^{a_s^{(u)}}\right)\left(\prod_{\substack {s<r \\ s \in I}}q_s^{-a_s^{(u)}}\right)\gamma_rq_r^{a^{(u)}_r}\ \ \text{and}\ \ \Lambda_v=\displaystyle\left(\prod_{\substack {s<r \\ s \notin I}}
    q_s^{a_s^{(v)}}\right)\left(\prod_{\substack {s<r \\ s \in I}}q_s^{-a_s^{(v)}}\right)\gamma_rq_r^{a^{(v)}_r}\] respectively. We claim that $\Lambda_u\neq \Lambda_v$. Indeed, \[\Lambda_u=\Lambda_v\implies \gamma_rq_r^{a^{(u)}_r}=\gamma_rq_r^{a^{(v)}_r} \implies q_r^{\left(a_r^{(u)}-a_r^{(v)}\right)}=1 \implies l_r\mid \left(a_r^{(u)}-a_r^{(v)}\right),\] which is a contradiction. 
\par Now $pz_r-\Lambda_up$ is a non zero element in $P$ of smaller length than $p$. Hence by induction, it follows that $P$ contains a basis vector of the form $e(a_1,\cdots,a_n)$ and hence $P=M\left(\underline{\mu}(I,J),\underline{\gamma}(I)\right)$ by the actions of $x_i$ and $y_i$ for all $1\leq i \leq n$. This completes the proof.   
\end{proof}
\begin{coro}
Let $M\left(\underline{\mu}(I,J),\underline{\gamma}(I)\right)$ be a simple ${{A}_n^{\underline{q},\Lambda}}$-module. Then
\begin{itemize}
    \item [(1)] The set $I$ consists of indices $i\in [1,n]$ so that $x_i$ act as nilpotent operator on $M\left(\underline{\mu}(I,J),\underline{\gamma}(I)\right)$.
    \item [(2)] The set $J$ consists of indices $j\in [1,n]$ so that $y_j$ act as nilpotent operator on $M\left(\underline{\mu}(I,J),\underline{\gamma}(I)\right)$.
    \item [(3)] The module $M\left(\underline{\mu}(I,J),\underline{\gamma}(I)\right)$ is $z_i$-torsionfree for all $1\leq i\leq n$.
\end{itemize}
\end{coro}
\subsection{Classification of Simple Modules:}
Let us first recall the assumption (\ref{asm}) as well as the commutation relations and identities for the algebra ${{A}_n^{\underline{q},\Lambda}}$ as mentioned at the beginning of this chapter. In the root of unity setting, the algebra ${{A}_n^{\underline{q},\Lambda}}$ is a prime affine PI algebra. Proposition \ref{sim} implies that the $\mathbb{K}$-dimension of each simple ${{A}_n^{\underline{q},\Lambda}}$-module is finite and bounded above by its PI degree $\prod\limits_{i=1}^{n}l_i$. Note that the action of each normal element $z_i$ on a simple ${{A}_n^{\underline{q},\Lambda}}$-module is either zero or invertible. 
\par Let $N$ be a simple ${{A}_n^{\underline{q},\Lambda}}$-module on which each $z_i$ acts invertibly. We can easily verify that the elements 
\begin{equation}\label{cev1}
   x_i^{l_i},\ y_i^{l_i},\  z_i\ \ \ \forall\ \ \ 1\leq i\leq n 
\end{equation}
in ${{A}_n^{\underline{q},\Lambda}}$ commute. Since $N$ has finite $\mathbb{K}$-dimension, there exists a common eigenvector $v$ in $N$ of the operators in (\ref{cev1}).
We define
\[
\begin{array}{ll}
vx_i^{l_i}=\alpha_iv, &\forall \ \ \ 1\leq i\leq n\\
vy_i^{l_i}=\beta_iv,&\forall \ \ \ 1\leq i\leq n\\
vz_i=\zeta_iv, &\forall \ \ \ 1\leq i\leq n
\end{array}
\]
for some $\alpha_i,\beta_i\in \mathbb{K}$ and $\zeta_i\in \mathbb{K}^*$. By Schur's lemma, each of the central elements $x_i^{l_i}$ and $y_i^{l_i}$ act as scalars on $N$. Thus, the operators $x_i$ or $y_i$ on $N$ are either invertible or nilpotent. Now set \[\mathbb{I}:=\{i:\alpha_i= 0\}\  \ \text{and}\ \ \mathbb{J}:=\{i:\beta_i=0\}.\] Clearly the set  $\mathbb{I}$ (respectively $\mathbb{J}$) consists of indices $i$ in $[1,n]$ for which the operators $x_i$ (respectively $y_i$) on $N$ are nilpotent. Depending on this, the following cases arise:\\
\textbf{Case I:} First consider the case $\mathbb{I}=\varphi$. Then $\alpha_i\neq 0$ for all $1\leq i\leq n$ and thus the operators $x_i$ on $N$ are invertible for all $1\leq i\leq n$. Then the vectors 
\[vx_n^{a_n}\cdots x_1^{a_1},\ 0\leq a_i\leq l_i-1\] are nonzero in $N$. Let us set for each $i\in [1,n]$:
\[\gamma_i:=\zeta_i\ \ \text{and}\ \  \mu_i:=l_i\text{-th root of}\ \alpha_i.\]
Thus $\underline{\mu}(\mathbb{I},\mathbb{J}):=(\mu_1,\cdots,\mu_n)\in (\mathbb{K}^{*})^n$ and $\underline{\gamma}(\mathbb{I}):=(\gamma_1,\cdots,\gamma_n)\in (\mathbb{K}^{*})^n$. Now define a $\mathbb{K}$-linear map 
\[\phi:M\left(\underline{\mu}(\mathbb{I},\mathbb{J}),\underline{\gamma}(\mathbb{I})\right) \longrightarrow N\]
by specifying the image of basis vectors:
\[
\phi\left(e(a_1,\cdots,a_n)\right):=\left(\prod\limits_{i=1}^{n}\mu_i^{-a_i}\right) vx_n^{a_n}\cdots x_1^{a_1}.
\]
We can easily check that $\phi$ is a nonzero ${{A}_n^{\underline{q},\Lambda}}(\mathbb{K})$-module homomorphism. Thus by Schur's lemma, $\phi$ is an isomorphism.\\
\noindent\textbf{Case II:} Next consider the case $\mathbb{I}\neq \varphi$. Then for each $i\in \mathbb{I}$, the operator $x_i$ on $N$ is nilpotent, and hence $\kere (x_i):=\{m\in N:mx_i=0\}$ is a nonzero subspace of $N$. We now claim that $\cap_{\substack{i\in \mathbb{I}}}\kere (x_i)\neq \{0\}$. Indeed if $i_1,i_2\in \mathbb{I}$, then the nilpotent operator $x_{i_2}$ on $N$ must keep $\kere(x_{i_1})$ invariant because of the commutation relation between $x_{i_1}$ and $x_{i_2}$. Hence $x_{i_2}$ is also nilpotent on the subspace $\kere(x_{i_1})$. That is $\kere(x_{i_1})\cap \kere(x_{i_2})\neq \{0\}$. Continuing this argument we can obtain $\cap_{i\in \mathbb{I}}\kere (x_i)\neq \{0\}$. \\
\textbf{Step 1:} Now each of the commuting operators in (\ref{cev1}) must keep the $\mathbb{K}$-space $\cap_{i\in \mathbb{I}}\kere (x_i)$ invariant due to commutation relations with $x_i$ for $i\in \mathbb{I}$. Therefore we can choose a common eigenvector $w\in \cap_{i\in \mathbb{I}}\kere (x_i)$ of the operators (\ref{cev1}). We define
\[\begin{array}{l}
wx_i=0,~i\in \mathbb{I}\\
wx^{l_i}_i=\xi_iw,~i\in [1,n]\setminus \mathbb{I}\\
wy^{l_i}_i=\eta_iw,~i\in [1,n]\\
wz_i=\zeta_iw,~i\in [1,n]
\end{array}\]
Clearly, we can observe that

\begin{itemize}
    \item $\xi_i\in \mathbb{K}^*$ for all $i\in [1,n]\setminus \mathbb{I}$ and $\zeta_i\in \mathbb{K}^*$ for all $i\in [1,n]$,
    \item $\eta_i=0$ if $i\in \mathbb{J}$ and $\eta_i\in \mathbb{K}^*$ if $i\in [1,n]\setminus \mathbb{J}$.
\end{itemize}
Now for each $i\in \mathbb{I}$, the relation $x_iy_i-q_iy_ix_i=z_{i-1}$ induces the following
\[
wz_{i-1}=w(x_iy_i-q_iy_ix_i)=-q_iwy_ix_i=q_iwz_i.\] Thus we obtain $q_i\zeta_i=\zeta_{i-1}$ for each $i\in \mathbb{I}$.\\
\textbf{Step 2:} We now claim that for each $i\in \mathbb{I}$, the vectors $wy_i^{a_i},~0\leq a_i\leq l_i-1$ in $N$ are nonzero. From the description of the set $\mathbb{J}$, if $i\in \mathbb{I}\setminus \mathbb{J}$ then $wy^{l_i}_i\neq 0$ and we are done. On the other hand if $\mathbb{I}\cap \mathbb{J}\neq \varphi$, then for each $i\in \mathbb{I}\cap \mathbb{J}$, we have $wy^{l_i}_i=0$. Suppose $k_i$ is the smallest index with $1\leq k_i\leq l_i$ such that $wy_i^{k_i-1}\neq 0\ \text{and}\ wy_i^{k_i}=0$. In fact, after simplifying the equality $wy^{k_i}x_i=0$ we obtain
\begin{align*}
    0=wy^{k_i}x_i&=q_i^{-k_i}w\left(x_iy_i^{k_i}-\displaystyle\frac{q_i^{k_i}-1}{q_i-1}z_{i-1}y_i^{k_i-1}\right)\\
    &=-q_i^{-k_i}\displaystyle\frac{q_i^{k_i}-1}{q_i-1}wz_{i-1}y_i^{k_i-1}\\
    &=-q_i^{-k_i}\displaystyle\frac{q_i^{k_i}-1}{q_i-1}\zeta_{i-1}wy_i^{k_i-1}.
\end{align*}
This implies $k_i$ is the smallest index such that $q_i^{k_i}=1$ and hence $k_i=l_i$. Thus we can conclude that for each $i\in \mathbb{I}$, the vectors $wy_i^{a_i},~0\leq a_i\leq l_i-1$ are nonzero in $N$.\\
\textbf{Step 3:} Let us set for each $i\in [1,n]$: 
\[\mu_i:=\begin{cases}
    l_i\text{-th root of}\ \xi_i,& i\in [1,n]\setminus \mathbb{I}\\
    l_i\text{-th root of}\ \eta_i, & i\in \mathbb{I}\setminus (\mathbb{I}\cap \mathbb{J})\\
    0,& i\in \mathbb{I}\cap \mathbb{J}
\end{cases}~ \ \text{and}\ \  
~\gamma_i:=\zeta_i.\]
Clearly $\underline{\mu}(\mathbb{I},\mathbb{J}):=(\mu_1,\cdots,\mu_n)\in \mathbb{K}^n$ with $\mu_i=0$ iff $i\in \mathbb{I}\cap \mathbb{J}$ and $\underline{\gamma}(\mathbb{I}):=(\gamma_1,\cdots,\gamma_n)\in (\mathbb{K}^{*})^{n}$ with $q_i\gamma_i=\gamma_{i-1}$ if $i\in \mathbb{I}$. Any ${{A}_n^{\underline{q},\Lambda}}$-module homomorphism must map $z_i$-eigenvectors of $M\left(\underline{\mu}(\mathbb{I},\mathbb{J}),\underline{\gamma}(\mathbb{I})\right)$ to $z_i$-eigenvectors of $N$
with the same eigenvalue. Define a $\mathbb{K}$-linear map
\[\psi:M\left(\underline{\mu}(\mathbb{I},\mathbb{J}),\underline{\gamma}(\mathbb{I})\right) \longrightarrow N\]
by setting
\begin{equation*}
\psi\left(e(a_1,\cdots,a_n)\right):=\displaystyle wf(a_n)\cdots f(a_1),
\end{equation*} 
where \[f(a_i)=\begin{cases}
\mu_i^{-a_i} x_i^{a_i}, & i\in [1,n]\setminus \mathbb{I}\\
 \mu_i^{-a_i} y_i^{a_i}, & i\in \mathbb{I}\setminus (\mathbb{I}\cap \mathbb{J})\\
 y_i^{a_i},& i\in \mathbb{I}\cap \mathbb{J}.
\end{cases}\]
It can be easily verified that $\psi$ is a nonzero ${{A}_n^{\underline{q},\Lambda}}$-module homomorphism. By Schur's lemma, $\psi$ is an isomorphism. 
\par Thus the above discussions lead us to the main result of this section which provides a classification of simple ${{A}_n^{\underline{q},\Lambda}}$-modules in terms of scalar parameters:
\begin{theo}\label{ztorsim}
 Each $z_i$-torsionfree for all $1\leq i\leq n$ simple ${{A}_n^{\underline{q},\Lambda}}$-module is isomorphic to a simple ${{A}_n^{\underline{q},\Lambda}}$-module $M\left(\underline{\mu}(I,J),\underline{\gamma}(I)\right)$ for two subsets $I$ and $J$ of $[1,n]$ with $\underline{\mu}(I,J):=(\mu_1,\cdots,\mu_n)\in \mathbb{K}^n$ such that $\mu_i=0$ iff $i\in I\cap J$ and $\underline{\gamma}(I):=(\gamma_1,\cdots,\gamma_n)\in \mathbb{({K}^*)}^{n}$ such that $q_i\gamma_i=\gamma_{i-1}$ if $i\in I$.  
\end{theo}
In view of this section together with Theorem \ref{onepart}, we have successfully achieved a comprehensive classification of maximal dimensional simple modules for ${{A}_n^{\underline{q},\Lambda}}$. Specifically, we have proved the following:
\begin{theo}\label{otm}
A simple ${{A}_n^{\underline{q},\Lambda}}$-module $M$ is $z_i$-torsionfree for all $1\leq i\leq n$ if and only if it is maximal dimensional, i.e., $\dime_{\mathbb{K}}(M)=\prod\limits_{i=1}^nl_i$.  
\end{theo} 
In the next sections, we aim to determine the center and the Azumaya locus for the quantum Weyl algebra under the roots of unity assumption (\ref{asm}).
\section{{The Center for Quantum Weyl Algebras}}\label{cenqwa}
In this section, we describe the centers of the PI quantized Weyl algebras under the assumption (\ref{asm}). Let $Z(A)$ denote the center of the algebra $A$.
\begin{prop}\label{center}
The center $Z({{A}_n^{\underline{q},\Lambda}})$ is the subalgebra generated by  $x_1^{l_1},y_1^{l_1},\cdots,x_n^{l_n},y_n^{l_n}$ and is
isomorphic to a polynomial algebra \[Z({{A}_n^{\underline{q},\Lambda}})=\mathbb{K}[x_1^{l_1},y_1^{l_1},\cdots,x_n^{l_n},y_n^{l_n}].\]
\end{prop}
\begin{proof} 
We proceed by induction $n$. The base case $n=1$ is clear from \cite[Lemma 2.2]{am}. Fix $n\geq 1$ and let $S$ denote the subalgebra of ${{A}_n^{\underline{q},\Lambda}}$ generated by the elements $x_1^{l_1},y_1^{l_1},\cdots,x_n^{l_n},y_n^{l_n}$. By Proposition \ref{l1} and the defining relations of ${{A}_n^{\underline{q},\Lambda}}$ along with the assumptions (\ref{asm}), the algebra $S$ is contained in the center of ${{A}_n^{\underline{q},\Lambda}}$ and is isomorphic to the polynomial algebra in variables $x_1^{l_1},y_1^{l_1},\cdots,x_n^{l_n},y_n^{l_n}$. Thus it suffices to show that $Z({{A}_n^{\underline{q},\Lambda}})$ is contained in $S$. Let $w\in Z({{A}_n^{\underline{q},\Lambda}})$. By the $\mathbb{K}$-basis (\ref{kbasis}), we can write
\[w=\sum\limits_{a,b}w(a,b)x_n^ay_n^b,\]
where the summation ranges over non-negative integers $a,b$ and each $w(a,b)$ is an element in the subalgebra ${{A}_{n-1}^{\underline{q},\Lambda}}$ generated by the $x_i,y_i$ for all $1\leq i\leq n-1$. The inclusion $Z({{A}_n^{\underline{q},\Lambda}})\subset S$ follows from the following claim, whose proof is divided into three steps:\\
\textbf{Claim:} The element $w(a,b)$ belongs to $S$ if $l_n$ divides both $a$ and $b$. Otherwise $w(a,b)=0$. \\
\textbf{Step 1:} We argue that $x_n$ and $y_n$ commutes with each $w(a,b)$. First note that $z_{i}$ commutes with each $w(a,b)$ for all $1\leq i\leq n-1$. This fact is a consequence of the observations that $w$ is central and $z_{i}$ commutes with $x_n$ and $y_n$ for all $1\leq i\leq n-1$. Next fix $a,b$ and abbreviate $w(a,b)$ by $w'$. Write
\[w'=\sum\limits_{{\underline{a},\underline{b}}}c_{\underline{a},\underline{b}}x^{\underline{a}}y^{\underline{b}}\]
where the sum ranges over pairs of $(n-1)$-tuples $\underline{a}=(a_1,\cdots,a_{n-1})$ and $\underline{b}=(b_1,\cdots,b_{n-1})$, and $x^{\underline{a}}y^{\underline{b}}$ abbreviates the element $x_1^{a_1}\cdots x_{n-1}^{a_{n-1}}y_1^{a_1}\cdots y_{n-1}^{a_{n-1}}$. Since $z_{i}$ commutes with $w'$ for all $1\leq i\leq n-1$, we use Proposition \ref{l1} to obtain
\[\sum\limits_{{\underline{a},\underline{b}}}c_{\underline{a},\underline{b}}x^{\underline{a}}y^{\underline{b}}z_{i}=w'z_{i}=z_{i}w'=\sum\limits_{{\underline{a},\underline{b}}}\prod\limits_{j=1}^{i}q_j^{(b_j-a_j)}c_{\underline{a},\underline{b}}x^{\underline{a}}y^{\underline{b}}z_{i}.\]
It follows that if $c_{\underline{a},\underline{b}}\neq 0$ then $\prod\limits_{j=1}^{i}q_j^{(b_j-a_j)}=1$ for all $1\leq i\leq n-1$. This implies $l_i$  divides $(b_i-a_i)$ for all $1\leq i\leq n-1$. Consequently, using the defining relations of ${{A}_n^{\underline{q},\Lambda}}$, we conclude that
\[x_nw'=\sum\limits_{{\underline{a},\underline{b}}}c_{\underline{a},\underline{b}}x_nx^{\underline{a}}y^{\underline{b}}=\sum\limits_{{\underline{a},\underline{b}}}c_{\underline{a},\underline{b}}\prod\limits_{i=1}^{n-1}(q_i\lambda_{in})^{(b_i-a_i)}x^{\underline{a}}y^{\underline{b}}x_n=w'x_n.\] An analogous argument shows that $y_nw'=w'y_n$.\\
\textbf{Step 2:} We argue that $w(a,b)=0$ if $l_n$ does not divide both $a$ and $b$. We have that
\[x_nw=\sum\limits_{a,b}w(a,b)x_n^{a+1}y_n^b=\sum\limits_{a\geq 1,b\geq 0}w(a-1,b)x_n^ay_n^b\] 
and 
\[\begin{array}{ll}
wx_n&=\left(\sum\limits_{a,b}w(a,b)x_n^ay_n^b\right)x_n\\
&=\sum\limits_{a,b}q_n^{-b}w(a,b)x_n^{a+1}y_n^b-\frac{1-q_n^{-b}}{1-q_n}z_{n-1}w(a,b)x_n^{a}y_n^{b-1}\\
&=\sum\limits_{a\geq 1,b\geq 0}q_n^{-b}w(a-1,b)x_n^{a}y_n^b-\sum\limits_{a,b}\frac{1-q_n^{-(b+1)}}{1-q_n}z_{n-1}w(a,b+1)x_n^{a}y_n^{b}.
\end{array}\]
Since $x_nw=wx_n$, we deduce that 
\[w(a-1,b)=q_n^{-b}w(a-1,b)+\frac{1-q_n^{-(b+1)}}{1-q_n}z_{n-1}w(a,b+1)\]
for $a\geq 1$ and $b\geq 0$. Rearranging this equation, it is straightforward to verify that this identity gives
\[(1-q_n^{-b})w(a,b)=\frac{1-q_n^{-(b+i)}}{(1-q_n)^i}z_{n-1}^iw(a+i,b+i)\]
for $a,b,i\geq 0$. Suppose that $l_n$ does not divide $b$, so that $q_n^{b}-1\neq 0$. There is a unique $i\in \{1,\cdots,l_n-1\}$ such that $l_n$ divides $b+i$. For this choice of $i$, we have 
\[w(a,b)=\frac{(1-q_n^{-(b+i)})z_{n-1}^iw(a+i,b+i)}{(1-q_n^{-b})(1-q_n)^i}=0.\]
Analogous argument using the identity $y_nw=wy_n$ shows that $w(a,b)=0$ if $l_n$ does not divide $a$.\\
\textbf{Step 3:} Finally, we show that if $l_n$ divides both $a$ and $b$, then $w(a,b)$ belongs to $S$. Since $w(a,b)$ lies in the subalgebra ${{A}_{n-1}^{\underline{q},\Lambda}}$, by induction it suffices to show that it lies in the center. This follows from the fact $x_i$ and $y_i$ commute with $w$ for $1\leq i\leq n-1$ and that $l_n$ divides both $a$ and $b$ whenever $w(a,b)$ is non-zero. This completes the proof.
\end{proof}
\begin{prop}\label{zre}
The following identities hold in ${{A}_{n}^{\underline{q},\Lambda}}$, for $i=1,\cdots,n$ and $z_0=1$:
    \[z_{i}^{l_i}=z_{i-1}^{{l_i}}+q_i^{l_i(l_i-1)/2}(q_i-1)^{l_i}y_i^{l_i}x_i^{l_i}\]
\end{prop}
\begin{proof}
We claim that the following formula holds in ${{A}_{n}^{\underline{q},\Lambda}}$, for any $k\geq 1$:
 \begin{equation}\label{identity}
 \left((q_i-1)y_ix_i+z_{i-1}\right)^k=q_i^{k(k-1)/2}(q_i-1)^ky_i^kx_i^k+\sum\limits_{r=1}^{k-1}c_r^{(k)}y_i^rx_i^r+z_{i-1}^{k}
\end{equation}
where $c_r^{(k)}$ are certain element of ${{A}_{n-1}^{\underline{q},\Lambda}}$. This claim can be easily verified by induction, using the following fact:
 \begin{enumerate}
     \item The element $z_{i-1}$ commutes with each of $x_i$ and $y_i$.
     \item For any $r\geq 1$, we have $(y_i^rx_i^r)(y_ix_i)=q_i^ry_i^{r+1}x_i^{r+1}+\displaystyle\frac{q_i^r-1}{q_i-1}z_{i-1}y_i^rx_i^r$.
 \end{enumerate}
In particular the above identity (\ref{identity}) for $k=l_i$ becomes 
\[z_i^{l_i}=q_i^{l_i(l_i-1)/2}(q_i-1)^{l_i}y_i^{l_i}x_i^{l_i}+\sum\limits_{r=1}^{{l_i}-1}c_r^{({l_i})}y_i^rx_i^r+z_{i-1}^{l_i}.\]
This implies \[\sum\limits_{r=1}^{{l_i}-1}c_r^{({l_i})}y_i^rx_i^r=z_i^{l_i}-q_i^{l_i(l_i-1)/2}(q_i-1)^{l_i}y_i^{l_i}x_i^{l_i}-z_{i-1}^{l_i}.\]
The right-hand side is central, so it follows that the left-hand side is also central. By Proposition \ref{center}, we conclude that $c_r^{(l_i)}=0$ for $1\leq r\leq l_i-1$. Thus we obtain
\[z_i^{l_i}=z_{i-1}^{l_i}+q_i^{l_i(l_i-1)/2}(q_i-1)^{l_i}y_i^{l_i}x_i^{l_i},\ \ 1\leq i\leq n.\]
\end{proof}
Let $S$ be the multiplicative set generated by the elements $z_i$ for $1 \leq i \leq n$. Since each $z_i$ is a normal element, $S$ forms an Ore set. We denote the localization of ${{A}_n^{\underline{q},\Lambda}}$ with respect to the Ore set $S$ as ${{B}_n^{\underline{q},\Lambda}}:={{A}_n^{\underline{q},\Lambda}}[S^{-1}]$. Then the center of ${{B}_n^{\underline{q},\Lambda}}$ is obtained by localizing $Z({{A}_n^{\underline{q},\Lambda}})$ with respect to the multiplicative set $S_c$ generated by the elements $z_i^{l_i}$ for $1 \leq i \leq n$. Therefore \[Z({{B}_n^{\underline{q},\Lambda}})=\mathbb{K}[x_1^{l_1},y_1^{l_1},\cdots,x_n^{l_n},y_n^{l_n}][S_c^{-1}].\]
Now it follows from Corollary \ref{pisame} that the algebra ${{B}_n^{\underline{q},\Lambda}}$ is a PI algebra with $\pideg ({{B}_n^{\underline{q},\Lambda}})=\pideg({{A}_n^{\underline{q},\Lambda}})$.
\section{The Azumaya Locus for Quantum Weyl Algebras}\label{azuqwa}
Assume the assumption (\ref{asm}) on the defining multiparameters of ${{A}_{n}^{\underline{q},\Lambda}}$. Then the algebra ${{A}_{n}^{\underline{q},\Lambda}}$ is prime affine PI algebra satisfying the hypothesis (H) with $\pideg{{A}_n^{\underline{q},\Lambda}}=\prod\limits_{i=1}^nl_i$. This allows us to define the Azumaya locus of ${{A}_{n}^{\underline{q},\Lambda}}$, which parameterizes the simple ${{A}_{n}^{\underline{q},\Lambda}}$-modules of maximal dimension. As $\mathbb{K}$ is algebraically closed, it is clear from Proposition \ref{center}, the maximal ideals of $Z({{A}_{n}^{\underline{q},\Lambda}})$ are of the form \[\mathfrak{m}:=\langle x_i^{l_i}-\alpha_i,y_i^{l_i}-\beta_i: 1\leq i\leq n\rangle\] for some scalars $\alpha_i,\beta_i\in\mathbb{K}$. Suppose \[\chi_{\mathfrak{m}}:Z({{A}_{n}^{\underline{q},\Lambda}})\rightarrow \mathbb{K}\] is the central character corresponding to $\mathfrak{m}\in Z({{A}_{n}^{\underline{q},\Lambda}})$. Then $\chi_{\mathfrak{m}}(x_i^{l_i})=\alpha_i$ and $\chi_{\mathfrak{m}}(y_i^{l_i})=\beta_i$. As each $z_i^{l_i}$ is central, therefore using Proposition \ref{zre}, we can write
\[\chi_{\mathfrak{m}}(z_i^{l_i})=\chi_{\mathfrak{m}}(z_{i-1}^{l_i})+q_i^{l_i(l_i-1)/2}(q_i-1)^{l_i}\beta_i\alpha_i\ \ \text{for}\ \ 1\leq i\leq n.\] By Theorem \ref{otm}, a simple ${{A}_{n}^{\underline{q},\Lambda}}$-module $N$ is maximal dimensional if and only if the action of each $z_i$ on $N$ is invertible. Therefore a simple ${{A}_{n}^{\underline{q},\Lambda}}$-module is maximal dimensional if and only if it is a simple module over ${{B}_{n}^{\underline{q},\Lambda}}$. Consequently, the Azumaya locus of ${{A}_{n}^{\underline{q},\Lambda}}$ is given by 
\begin{align*}
\mathcal{AL}\left({{A}_{n}^{\underline{q},\Lambda}}\right)&=\{\mathfrak{m}\in \mspect Z({{A}_{n}^{\underline{q},\Lambda}}):\chi_{\mathfrak{m}}(z_i^{l_i})\neq 0,~\forall~1\leq i\leq n\}\\
&=\mspect Z({{B}_{n}^{\underline{q},\Lambda}}).
\end{align*}
Thus we have established the following
\begin{theo}\label{qwaal}
The Azumaya locus of ${{A}_{n}^{\underline{q},\Lambda}}$ is $\mspect Z({{B}_{n}^{\underline{q},\Lambda}})$. Moreover, ${{B}_{n}^{\underline{q},\Lambda}}$ is Azumaya algebra over $Z({{B}_{n}^{\underline{q},\Lambda}})$. 
\end{theo}
\section{Alternative Quantum Weyl Algebras}\label{altqwa}
Another family of multiparameter quantized Weyl algebras has been studied in the literature, including \cite{aj}. This family exhibits more symmetric defining relations compared to those of ${{A}_n^{\underline{q},\Lambda}}$. Given such $\Lambda$ and $\underline{q}$, the multiparameter quantized Weyl algebra ${\mathcal{A}_n^{\underline{q},\Lambda}}$ of symmetric type due to  M. Akhavizadegan and D. Jordan, called as {\it alternative quantum Weyl algebra}, is the algebra generated over the field $\mathbb{K}$ by the variables $x_1,y_1,\cdots,x_n,y_n$ subject to the following relations:
\begin{align*}
&y_iy_j=\lambda_{ij}y_jy_i \ \ \ \ \ \forall \ \ \ 1\leq i<j\leq n,\\
&x_ix_j=\lambda_{ij}x_jx_i \ \ \ \ \ \forall \ \ \ 1\leq i<j\leq n,\\
&x_iy_j=\lambda_{ij}^{-1}y_jx_i \ \ \ \ \ \forall \ \ \ 1\leq i<j\leq n,\\
&y_ix_j=\lambda_{ij}^{-1}x_jy_i \ \ \ \ \ \forall \ \ \ 1\leq i<j\leq n,\\
&x_iy_i-q_iy_ix_i=1\ \ \ \ \ \forall \ \ \ 1\leq i \leq n.
\end{align*}  
The algebra ${\mathcal{A}_n^{\underline{q},\Lambda}}$ has an iterated skew polynomial presentation with respect to the ordering of indices $y_1,x_1,\cdots,y_n,x_n$ of the form: 
$$\mathbb{K}[y_1][x_1,\tau_1,\delta_1][y_2,\sigma_2][x_2,\tau_2,\delta_2]\cdots [y_n,\sigma_n][x_n,\tau_n,\delta_n]$$
where the $\tau_j$ and $\sigma_{j}$ are $\mathbb{K}$-linear automorphisms and the $\delta_j$ are $\mathbb{K}$-linear $\tau_j$-derivations  such that
\begin{align*}
  \tau_j(y_i)&=\lambda_{ij}y_i,\ i<j & \sigma_j(y_i)&=\lambda_{ij}^{-1}y_i,\ i<j\\
  \tau_j(x_i)&=\lambda_{ij}^{-1}x_i,\ i<j& \sigma_j(x_i)&=\lambda_{ij}x_i,\ i<j\\
  \tau_j(y_j)&=q_jy_j,\ \forall~j & 
  \delta_j(x_i)&=\delta_j(y_i)=0,\ i<j\\
  \delta_j(y_j)&=1,\ \forall ~j&&
\end{align*}
Using the skew Hilbert basis theorem, we can conclude that the algebra ${\mathcal{A}_n^{\underline{q},\Lambda}}$ is an affine noetherian domain. Moreover, the family \[\{y_1^{a_1}x_1^{b_1}\cdots y_n^{a_n}x_n^{b_n}~:~a_i,b_i\in \mathbb{Z}_{\geq 0}\}\] is a $\mathbb{K}$-basis of ${\mathcal{A}_n^{\underline{q},\Lambda}}$. Similar to the quantum Weyl algebra, in the context of the alternative quantum Weyl algebra, we define $z_0=1$ and 
\[z_i=x_iy_i-y_ix_i\ \  \forall\ \  1\leq i\leq n.\] These elements will also play a crucial role in this section. It can be easily verified that for all $1\leq i\leq n$,
\[{z}_i=1+(q_i-1)y_ix_i=q_i^{-1}\left(1+(q_i-1)x_iy_i\right).\]
The following proposition is well known (cf. \cite{aj}) for the algebra ${\mathcal{A}_n^{\underline{q},\Lambda}}$ and can be proved using the defining relations.
\begin{prop} \label{l11}
Direct computations yield the following results for ${\mathcal{A}_n^{\underline{q},\Lambda}}$.
\begin{enumerate}
\item For all $1\leq i\leq n$, $z_i$ is a normal element of ${\mathcal{A}_n^{\underline{q},\Lambda}}$. More precisely, we have: 
\begin{itemize}
    \item [(a)] For all $i,j,~~1\leq i\neq j\leq n$, $z_ix_j=x_jz_i$ and $z_iy_j=y_jz_i$ 
    \item [(b)] For all $i,~~1\leq i\leq n$, $z_ix_i=q_i^{-1}x_iz_i$ and $z_iy_i=q_iy_iz_i$
    \item [(c)] For all $i,j$ with $1\leq i,j\leq n$, $z_iz_j=z_jz_i$.
\end{itemize}
\item For $k\geq 1$, the following identities hold in the algebra ${\mathcal{A}_n^{\underline{q},\Lambda}}$
\begin{itemize}
    \item[(a)] $x_i^ky_i=q_i^ky_ix_i^k+\left(1+q_i+\cdots+q_i^{k-1}\right)x_i^{k-1}$,
    \item[(b)] $x_iy_i^k=q_i^ky_i^kx_i+\left(1+q_i+\cdots+q_i^{k-1}\right)y_i^{k-1}$. 
\end{itemize}
\end{enumerate}
\end{prop}
In this section, our objective is to establish
analogous results for ${\mathcal{A}_n^{\underline{q},\Lambda}}$ in parallel to those obtained for ${{A}_n^{\underline{q},\Lambda}}$. To begin, we derive the following necessary and sufficient conditions for the alternative quantum Weyl algebra to be a PI algebra.
\begin{prop}
The alternative quantum Weyl algebra ${\mathcal{A}_n^{\underline{q},\Lambda}}$ is a PI algebra if and only if the parameters $q_i$ and $\lambda_{ij}$ are roots of unity.
\end{prop}
\begin{proof}
The argument is parallel to Proposition \ref{qwapia}, based on the defining relations of ${\mathcal{A}_n^{\underline{q},\Lambda}}$ and the identities in Proposition \ref{l11}.
\end{proof}
We denote by ${\mathcal{B}_n^{\underline{q},\Lambda}}$ the localization of ${\mathcal{A}_n^{\underline{q},\Lambda}}$ with respect to the Ore set generated by the normal elements $z_i,\ 1\leq i\leq n$. Now we recall a result from \cite{aj} which states that the localizations of both versions of quantum Weyl algebras are isomorphic. 
\begin{theo}\label{loiso}\emph{(\cite[1.7]{aj})}
The map $\theta:{\mathcal{B}_n^{\underline{q},\Lambda}}\longrightarrow{{B}_n^{\underline{q},\Lambda}}$ defined by \[\theta(y_i)=y_i, \theta(x_i)=z_{i-1}^{-1}x_i\] with $z_0=1$ is an isomorphism of $\mathbb{K}$-algebras.
\end{theo}
The above isomorphism will play a crucial role to obtain the PI degree, maximal dimensional simple module, center, and the Azumaya locus for ${\mathcal{A}_n^{\underline{q},\Lambda}}$.
\subsection{PI degree for \texorpdfstring{${\mathcal{A}_n^{\underline{q},\Lambda}}$}{TEXT}} Here we wish to compute the PI degree of ${\mathcal{A}_n^{\underline{q},\Lambda}}$, taking into account the roots of unity assumption (\ref{asm}) concerning the defining multiparameters ${q_i}$ and $\lambda_{ij}$. In the roots of unity setting, it follows from Corollary \ref{pisame} that the algebras ${B}_n^{\underline{q},\Lambda}$ and $\mathcal{B}_n^{\underline{q},\Lambda}$ are PI algebras and \[\pideg ({{A}_n^{\underline{q},\Lambda}})=\pideg ({B}_n^{\underline{q},\Lambda}),\ \pideg ({\mathcal{A}_n^{\underline{q},\Lambda}})=\pideg (\mathcal{B}_n^{\underline{q},\Lambda}).\]
Therefore using the isomorphism in Theorem \ref{loiso}
we obtain \[\pideg ({A}_n^{\underline{q},\Lambda})=\pideg (\mathcal{A}_n^{\underline{q},\Lambda}).\]
Consequently, based on Theorem \ref{pitheorem}, we have now established that
\begin{theo}
The PI degree of ${A}_n^{\underline{q},\Lambda}$ is equal to the PI degree of ${\mathcal{B}_n^{\underline{q},\Lambda}}$, and both are given by
\begin{equation}\label{aqwapideg}
\pideg{\mathcal{A}_n^{\underline{q},\Lambda}}=\pideg {\mathcal{B}_n^{\underline{q},\Lambda}}=\prod\limits_{i=1}^nl_i.
\end{equation}
\end{theo}
\begin{rema}\normalfont
It is worth noting that there exists a simple $\mathcal{A}_n^{\underline{q},\Lambda}$-module with $\mathbb{K}$-dimension $\prod\limits_{i=1}^nl_i$. Subsection \ref{maxaqwa} will present an explicit construction and classification of these simple modules.   
\end{rema}
\subsection{PI degree for Prime Factors of \texorpdfstring{${\mathcal{A}_n^{\underline{q},\Lambda}}$}{TEXT}}
Fix an index $r$ with $1\leq r\leq n$. Then the ideal generated by the normal element $z_r:=x_ry_r-y_rx_r$ is a completely prime ideal of ${\mathcal{A}_n^{\underline{q},\Lambda}}$ (see \cite[Proposition 4.5]{aj}) and hence the factor algebra ${\mathcal{A}_n^{\underline{q},\Lambda}}/\langle z_r \rangle$ is a domain. Note that the algebra ${\mathcal{A}_n^{\underline{q},\Lambda}}/\langle z_r \rangle$ is a prime affine PI algebra under the root of unity assumption (\ref{asm}). Here we wish to compute the PI degree of the factor algebra ${\mathcal{A}_n^{\underline{q},\Lambda}}/\langle z_r \rangle$ explicitly.\\
\textbf{Step 1:} First we consider an algebra that is similar to Step 1 in Section \ref{sr}. Let $\mathcal{R}_r$ denote the $\mathbb{K}$-algebra generated by the variables $X_1,Y_1,\cdots,X_n,Y_n$ subject to the relations:
\begin{align}\label{qwadr2}
Y_iY_j&=\lambda_{ij}Y_jY_i \ \ \ \ \ \forall \ \ \ 1\leq i<j\leq n\nonumber\\
X_iX_j&=\lambda_{ij}X_jX_i \ \ \ \ \ \forall \ \ \ 1\leq i<j\leq n\nonumber\\
X_iY_j&=\lambda_{ij}^{-1}Y_jX_i\ \ \ \ \ \forall \ \ \ 1\leq i<j\leq n\\
Y_iX_j&=\lambda_{ij}^{-1}X_jY_i \ \ \ \ \ \forall \ \ \ 1\leq i<j\leq n\nonumber\\
X_rY_r&=Y_rX_r\nonumber\\
X_iY_i&-q_iY_iX_i=1 \ \ \ \ \ \forall \ \ \ 1\leq i\neq r\leq n\nonumber
\end{align}
In the root of unity setting (\ref{asm}), the algebra $\mathcal{R}_r$ becomes a PI algebra. The following result provides a connection between the algebras ${\mathcal{A}_n^{\underline{q},\Lambda}}$ and $\mathcal{R}_{r}$.
\begin{lemm}\label{cpfqwa1}
The factor algebra ${\mathcal{A}_n^{\underline{q},\Lambda}}/\langle z_r \rangle$ is isomorphic to the factor algebra $\mathcal{R}_{r}/\mathcal{I}_{r}$, where $\mathcal{I}_r$ is the ideal of $\mathcal{R}_r$ generated by $1+(q_r-1)Y_rX_r$. 
\end{lemm}
The proof of this lemma follows a similar approach as that of Lemma \ref{cpfqwa}. As a result, based on Remark \ref{piquotient}, we can conclude that 
\begin{equation}\label{noteq1}
    \pideg ({\mathcal{A}_n^{\underline{q},\Lambda}}/\langle z_r \rangle)=\pideg{(\mathcal{R}_{r}/\mathcal{I}_{r})}\leq \pideg{(\mathcal{R}_r)}.
 \end{equation}  
\noindent\textbf{Step 2:} Our objective now is to determine the PI degree of the algebra $\mathcal{R}_r$. By following the same reasoning outlined in Step 2 of Section \ref{sr}, we can establish that \[\pideg(\mathcal{R}_r)=\prod\limits_{\substack{i=1\\i\neq r}}^{n} l_i.\] Consequently, using (\ref{noteq1}), we obtain the inequality
\begin{equation}\label{leqaaq}
\pideg ({\mathcal{A}_n^{\underline{q},\Lambda}}/\langle z_r \rangle)\leq 
\displaystyle\prod\limits_{\substack{i=1\\i\neq r}}^{n} l_i.
\end{equation}
\noindent\textbf{Step 3:} Finally we aim to establish the equality in (\ref{leqaaq}). It is important to note that the alternative quantum Weyl algebra ${\mathcal{A}_{n-1}^{\underline{q}(r),\Lambda(r)}}$ can be viewed as a subalgebra of the factor algebra ${\mathcal{A}_n^{\underline{q},\Lambda}}/\langle z_r \rangle$. Here $\underline{q}(r)$ is the $(n-1)$-tuple elements of $\mathbb{K}\setminus\{0,1\}$ obtained from $\underline{q}$ by removing the $r$-th element and $\Lambda(r)$ is the $(n-1)\times (n-1)$ multiplicatively antisymmetric matrix obtained by simultaneously removing the $r$-th row and column from $\Lambda$. Thus from Remark \ref{piquotient} we have:
\begin{equation}
\pideg ({\mathcal{A}_{n-1}^{\underline{q}(r),\Lambda(r)}})\leq \pideg({\mathcal{A}_n^{\underline{q},\Lambda}}/\langle z_r \rangle)
\end{equation}
As a consequence of (\ref{aqwapideg}), we can establish that 
\begin{equation}
\pideg ({\mathcal{A}_{n-1}^{\underline{q}(r),\Lambda(r)}})=\displaystyle\prod\limits_{\substack{i=1\\i\neq r}}^{n} l_i.
\end{equation}
Consequently, we have established the following
\begin{theo}\label{itsnew}
Assuming the root of unity assumption (\ref{asm}) on the defining multiparameters, the PI degree of a prime factor of $\mathcal{A}_n^{\underline{q},\Lambda}$ by the normal element $z_r$ is given by $\displaystyle\prod\limits_{\substack{i=1\\i\neq r}}^{n} l_i.$
\end{theo}
\noindent\textbf{Question:} In view of Proposition \ref{sim}, there exists a $z_r$-torsion simple ${\mathcal{A}_n^{\underline{q},\Lambda}}$-module with $\mathbb{K}$-dimension $\prod\limits_{\substack{i=1\\i\neq r}}^{n} l_i$ for each $1\leq r\leq n$. It would be interesting to explore the construction of such simple modules and subsequently classify them. So let us give it a try!
\subsection{Classification of Maximal Dimensional Simple Modules}\label{maxaqwa}
Here we classify all maximal dimensional simple ${\mathcal{A}_n^{\underline{q},\Lambda}}$-modules. As a direct consequence of Theorem \ref{loiso}, we can derive the following result, establishing a significant connection between the simple modules of both versions of quantum Weyl algebras.
\begin{prop}\label{conn}
There is a a one-to-one correspondence between the simple $z_i$-torsionfree ${{A}_n^{\underline{q},\Lambda}}$-modules and the simple $z_i$-torsionfree ${\mathcal{A}_n^{\underline{q},\Lambda}}$-modules.
\end{prop}
In light of this proposition along with Theorem \ref{ztorsim}, we can classify simple $z_i$-torsionfree $\mathcal{A}_n^{\underline{q},\Lambda}$-modules. In particular, each $z_i$-torsionfree for all $1\leq i\leq n$ simple ${\mathcal{A}_n^{\underline{q},\Lambda}}$-module is isomorphic to a simple ${{A}_n^{\underline{q},\Lambda}}$-module $M\left(\underline{\mu}(I,J),\underline{\gamma}(I)\right)$. Moreover, the $\mathbb{K}$-dimension of each simple $z_i$-torsionfree ${\mathcal{A}_n^{\underline{q},\Lambda}}$-module is equal to $\pideg{\mathcal{A}_n^{\underline{q},\Lambda}}$ which is maximal. 
\par Next, we aim to prove that each maximal dimensional (i.e., equal to the $\pideg {\mathcal{A}_n^{\underline{q},\Lambda}}$) simple ${\mathcal{A}_n^{\underline{q},\Lambda}}$-module is $z_i$-torsionfree for all $1\leq i\leq n$. If possible let $M$ be a maximal dimensional simple ${\mathcal{A}_n^{\underline{q},\Lambda}}$-module which is $z_r$-torsion for some $1\leq r\leq n$. As $z_r$ is a normal element then $M$ becomes a simple module over the factor algebra ${\mathcal{A}_n^{\underline{q},\Lambda}}/\langle z_r \rangle$. Hence by Proposition \ref{sim} along with Theorem \ref{itsnew}, we obtain \[\dime_{\mathbb{K}}(M)\leq \pideg({\mathcal{A}_n^{\underline{q},\Lambda}}/\langle z_r \rangle)<\pideg ({\mathcal{A}_n^{\underline{q},\Lambda}}).\] This is a contradiction. 
\par Consequently, we have achieved a comprehensive classification of maximal dimensional simple ${\mathcal{A}_n^{\underline{q},\Lambda}}$-modules. In particular, we have established the following. 
\begin{theo}\label{aqmd}
A simple ${\mathcal{A}_n^{\underline{q},\Lambda}}$-module $M$ is $z_i$-torsionfree for all $1\leq i\leq n$ if and only if it is maximal dimensional, i.e., $\dime_{\mathbb{K}}(M)=\prod\limits_{i=1}^nl_i$.
\end{theo}
\subsection{The Center and Azumaya Locus for \texorpdfstring{${\mathcal{A}_n^{\underline{q},\Lambda}}$}{TEXT}}
In this subsection, we describe $Z({\mathcal{A}_n^{\underline{q},\Lambda}})$ and determine the Azumaya locus. By the isomorphism $\theta$ in Theorem \ref{loiso}, we have 
\begin{equation}\label{calo}
Z({\mathcal{B}_n^{\underline{q},\Lambda}})=\theta^{-1}(Z({{B}_n^{\underline{q},\Lambda}}))=\theta^{-1}(\mathbb{K}[x_1^{l_1},y_1^{l_1},\cdots,x_n^{l_n},y_n^{l_n}][S_c^{-1}]).
\end{equation}
From the definition of $\theta$, it follows that 
\[\theta^{-1}(x_i^{l_i})=(z_1\cdots z_{i-1}x_i)^{l_i}, \ \theta^{-1}(y_i^{l_i})=y_i^{l_i},\ \theta^{-1}(z_i^{l_i})=(z_1\cdots z_{i-1}z_i)^{l_i}\]
and $\mathcal{S}_c:=\theta^{-1}(S_c)$ is also a multiplicative set in ${\mathcal{A}_n^{\underline{q},\Lambda}}$ contained in the multiplicative set $T_c$ generated by the $z_i^{l_i},\ 1\leq i\leq n$. Now applying $\theta^{-1}$ in the identity of Proposition \ref{zre}, we obtain the following identity for the central element $z_i^{l_i}$ in $\mathcal{A}_n^{\underline{q},\Lambda}$
\begin{equation}\label{id2a}
z_i^{l_i}=1+q_i^{l_i(l_i-1)/2}(q_i-1)^{l_i}y_i^{l_i}x_i^{l_i}, \ \ 1\leq i\leq n
\end{equation}
Therefore from (\ref{calo}),
\begin{align*}
Z({\mathcal{B}_n^{\underline{q},\Lambda}})&=\mathbb{K}[\theta^{-1}(x_i^{l_i}),\theta^{-1}(y_i^{l_i}):1\leq i\leq n][\mathcal{S}_c^{-1}]\\
&=\mathbb{K}[x_i^{l_i},y_i^{l_i}:1\leq i\leq n][T_c^{-1}].
\end{align*}
Thus the center of $\mathcal{A}_n^{\underline{q},\Lambda}$ is 
\[Z(\mathcal{A}_n^{\underline{q},\Lambda})=Z({\mathcal{B}_n^{\underline{q},\Lambda}})\cap \mathcal{A}_n^{\underline{q},\Lambda}=\mathbb{K}[x_i^{l_i},y_i^{l_i}:1\leq i\leq n].\]
As $\mathbb{K}$ is algebraically closed, the maximal ideals of $Z({\mathcal{A}_{n}^{\underline{q},\Lambda}})$ are of the form \[\mathfrak{m}:=\langle x_i^{l_i}-\alpha_i,y_i^{l_i}-\beta_i: 1\leq i\leq n\rangle\] for some scalars $\alpha_i,\beta_i\in\mathbb{K}$. Suppose \[\chi_{\mathfrak{m}}:Z({\mathcal{A}_{n}^{\underline{q},\Lambda}})\rightarrow \mathbb{K}\] is the central character corresponding to $\mathfrak{m}\in Z({\mathcal{A}_{n}^{\underline{q},\Lambda}})$. Then $\chi_{\mathfrak{m}}(x_i^{l_i})=\alpha_i$ and $\chi_{\mathfrak{m}}(y_i^{l_i})=\beta_i$. Hence from (\ref{id2a}), we can write
\[\chi_{\mathfrak{m}}(z_i^{l_i})=1+q_i^{l_i(l_i-1)/2}(q_i-1)^{l_i}\beta_i\alpha_i\ \ \text{for}\ \ 1\leq i\leq n.\] By Theorem \ref{aqmd}, a simple ${\mathcal{A}_{n}^{\underline{q},\Lambda}}$-module $N$ is maximal dimensional if and only if the action of each $z_i$ on $N$ is invertible. Consequently, the Azumaya locus of ${\mathcal{A}_{n}^{\underline{q},\Lambda}}$ is given by 
\begin{align*}
\mathcal{AL}\left({\mathcal{A}_{n}^{\underline{q},\Lambda}}\right)&=\{\mathfrak{m}\in \mspect Z({\mathcal{A}_{n}^{\underline{q},\Lambda}}):\chi_{\mathfrak{m}}(z_i^{l_i})\neq 0,~\forall~1\leq i\leq n\}\\
&=\mspect Z({\mathcal{B}_{n}^{\underline{q},\Lambda}}).
\end{align*}
Thus we have established the following
\begin{theo}
The center of $\mathcal{A}_n^{\underline{q},\Lambda}$ is 
$\mathbb{K}[x_i^{l_i},y_i^{l_i}:1\leq i\leq n]$ and the Azumaya locus of ${\mathcal{A}_{n}^{\underline{q},\Lambda}}$ is $\mspect Z({\mathcal{B}_{n}^{\underline{q},\Lambda}})$. Moreover, ${\mathcal{B}_{n}^{\underline{q},\Lambda}}$ is Azumaya algebra over $Z({\mathcal{B}_{n}^{\underline{q},\Lambda}})$. 
\end{theo}
Finally, it is worth mentioning that the Azumaya property of the algebra ${\mathcal{B}_{n}^{\underline{q},\Lambda}}$ can also be established by utilizing the isomorphism presented in Theorem \ref{loiso} along with Theorem \ref{qwaal}.
\section{Concluding Remarks} In conclusion, this chapter provides significant contributions to the study of multiparameter quantized algebras at roots of unity. We have achieved an explicit expression for the PI degree and the center, as well as a comprehensive classification of the maximal dimensional simple modules. Consequently, we have obtained the Azumaya locus of the multiparameter quantized algebras at roots of unity.
\par It is worth noting that recent studies by Brown and Yakimov \cite{bry} have shed light on the connection between the Azumaya locus $\mathcal{AL}(R)$ and the discriminant ideal $D_{l}$ of a prime PI algebra $R$. Specifically, in their main theorem \cite[Main Theorem]{bry}, they established that the zero set of $D_{l}$ is the complement of the Azumaya locus of $R$, in the case when $l$ is the square of the PI degree of $R$. As an application, the classification of the Azumaya locus of the multiparameter quantized algebras at roots of unity was obtained in \cite[Theorem 6.2]{bry} by determining its discriminant ideal.
\par Thus, in particular, this chapter provides an alternative approach to determine the Azumaya locus of the multiparameter quantized Weyl algebras without explicitly finding its discriminant ideal.

\section*{Acknowledgements}
The authors are grateful to the National Board of Higher Mathematics, Department of Atomic Energy, Government of India for providing financial support to carry out this research.

\end{document}